\documentclass[11pt]{article}
\usepackage{amsmath,amstext,amssymb,amsfonts,amsthm}
\usepackage{mathrsfs}
\usepackage{subfigure}
\usepackage{epsfig}
\usepackage[ruled]{algorithm2e}

\usepackage{graphicx}
\usepackage{accents}
\usepackage{color}
\usepackage[export]{adjustbox}
\usepackage{epstopdf}

\newtheorem{lem}{Lemma}[section]
\newtheorem{thm}{Theorem}[section]

\newtheorem{rem}{Remark}[section]
\newtheorem{Def}{Definiton}[section]
\newtheorem{exmp}{Example}

\numberwithin{equation}{section}
\numberwithin{figure}{section}

\newcommand{\caT}{\mathcal{T}}

\newcommand{\R}{\mathbb{R}}

\newcommand{\V}{\mathbb{V}}

\newcommand{\dist}{{\rm dist}}
\newcommand{\Om}{\Omega}
\newcommand{\Ga}{\Gamma}

\newcommand{\na}{\nabla}
\newcommand{\al}{\alpha}

\newcommand{\lj}{[{\hskip -1.5pt} [}
\newcommand{\rj}{]{\hskip -1.5pt} ]}

\newcommand{\be}{\begin{eqnarray}}
\newcommand{\ee}{\end{eqnarray}}
\newcommand{\ben}{\begin{eqnarray*}}
\newcommand{\een}{\end{eqnarray*}}
\newcommand{\beq}{\begin{equation}}
\newcommand{\eeq}{\end{equation}}

\newcommand{\cE}{\mathcal{E}}
\newcommand{\cM}{\mathcal{M}}

\newcommand{\cP}{\mathcal{P}}
\newcommand{\cS}{\mathcal{S}}
\newcommand{\cT}{\mathcal{T}}

\newcommand{\la}{\langle}
\newcommand{\ra}{\rangle}

\newcommand{\pa}{\partial}

\newcommand{\X}{\mathbb{X}}
\newcommand{\bv}{\boldsymbol{v}}
\newcommand{\bV}{\boldsymbol{V}}
\newcommand{\bW}{\boldsymbol{W}}

\newcommand{\rev}{}
\newcommand{\revto}{}
\newcommand{\nn}{\nonumber}

\allowdisplaybreaks[4]

\title{An Arbitrarily High Order Unfitted Finite Element Method for Elliptic Interface Problems with Automatic Mesh Generation\footnotemark[1]}
\author{
Zhiming Chen\footnotemark[2]
\and
Yong Liu\footnotemark[3]
}

\date{}

\begin{document}
\maketitle

\renewcommand{\thefootnote}{\fnsymbol{footnote}}
\footnotetext[1]{This work is supported in part by China National Key Technologies R\&D Program under the grant 2019YFA0709602 and China Natural Science Foundation under the grant 118311061,12288201.}
\footnotetext[2]{LSEC, Institute of Computational Mathematics,
Academy of Mathematics and System Sciences and School of Mathematical Science, University of
Chinese Academy of Sciences, Chinese Academy of Sciences,
Beijing 100190, China. E-mail: zmchen@lsec.cc.ac.cn}
\footnotetext[3]{LSEC, Institute of Computational Mathematics, Academy of Mathematics and Systems Science, Chinese Academy of Sciences, Beijing 100190, P.R. China.  E-mail: yongliu@lsec.cc.ac.cn}

\begin{center}
\small
\begin{minipage}{0.9\textwidth}
\textbf{Abstract.}
We consider the reliable implementation of high-order unfitted finite element methods on Cartesian meshes with hanging nodes for elliptic interface problems. We construct a reliable algorithm to merge small interface elements with their surrounding elements to automatically generate the finite element mesh whose elements are large with respect to both domains. We propose new basis functions for the interface elements to control the growth of the condition number of the stiffness matrix in terms of the finite element approximation order, the number of elements of the mesh, and the interface deviation which quantifies the mesh resolution of the geometry of the interface. Numerical examples are presented to illustrate the competitive performance of the method.

\medskip

\end{minipage}
\end{center}
\setlength{\parindent}{2em}
\section{Introduction}\label{sec_intro}

Interface problems arise from diverse physical and engineering applications in which the coefficients of the governing partial differential equations are discontinuous across material interfaces that separate the physical domains. The body-fitted finite element methods resolve the geometry of the interface by requiring the vertices of the finite element mesh located on the interfaces \cite{Babuska70, Chen98, ChenLong}. For domains with complex geometry, the construction of body-fitted shape regular finite element meshes may be difficult and time-consuming, which is the main driving force of the study of unfitted finite element methods. In this paper we will show that the shape regular body-fitted mesh can indeed be constructed for any shaped smooth interface based on our new merging cell algorithm  (see remarks below Theorem \ref{thm:3.1}). We emphasize, however, that even when the body-fitted shape regular mesh is available, the construction of high-order finite element methods still requires substantial new ideas including, for example, the isoparametric finite element method \cite{Ciarlet, Lehrenfeld} 
or unfitted finite element methods which are the focus of this paper. We remark that the shape regularity assumption of the finite element mesh is not only fundamental in the mathematical theory of finite element methods (see, e.g., \cite{Ciarlet}) but also essential in controlling the condition number of the finite element stiffness matrix for elliptic equations (see, e.g. \cite{Bank}).

Let $\Om\subset\R^2$ be a bounded Lipschitz domain which is divided by a $C^2$-smooth interface $\Ga$ into two nonintersecting subdomains $\Omega_1\subset\bar{\Omega}_1\subset \Omega$, $\Omega_2=\Omega\setminus\bar{\Omega}_1$, see
Fig.\ref{fig:1.1}. We consider the following elliptic interface problem
\begin{align}
&-{\rm div}(a\nabla u)=f\ \ \mbox{in }\Omega_1\cup\Omega_2,\label{m1}\\
&[\![u]\!]_{\Gamma}=0, \, [\![a\nabla u \cdot n]\!]_{\Gamma}=0\ \ \mbox{on } \Gamma,\ \ u=g\ \ \mbox{on }\pa\Om,\label{m2}
\end{align}
where $f\in L^2(\Omega)$, $g\in H^{1/2}(\partial \Omega)$, $n$ is the unit outer normal to $\Omega_1$, and $[\![v]\!]:=v|_{\Omega_1}-v|_{\Omega_2}$ stands for the jump of a function $v$ across the interface $\Gamma$. We assume that the coefficient $a(x)$ is positive and piecewise constant, namely,
$a=a_1\chi_{\Omega_1}+a_2\chi_{\Omega_2}$, $a_1,a_2 >0$, where $\chi_{\Omega_i}$ denotes the characteristic function of $\Omega_i$, $i=1,2$.

\begin{figure}[htbp]
  \centering
  \includegraphics[width=0.3\linewidth]{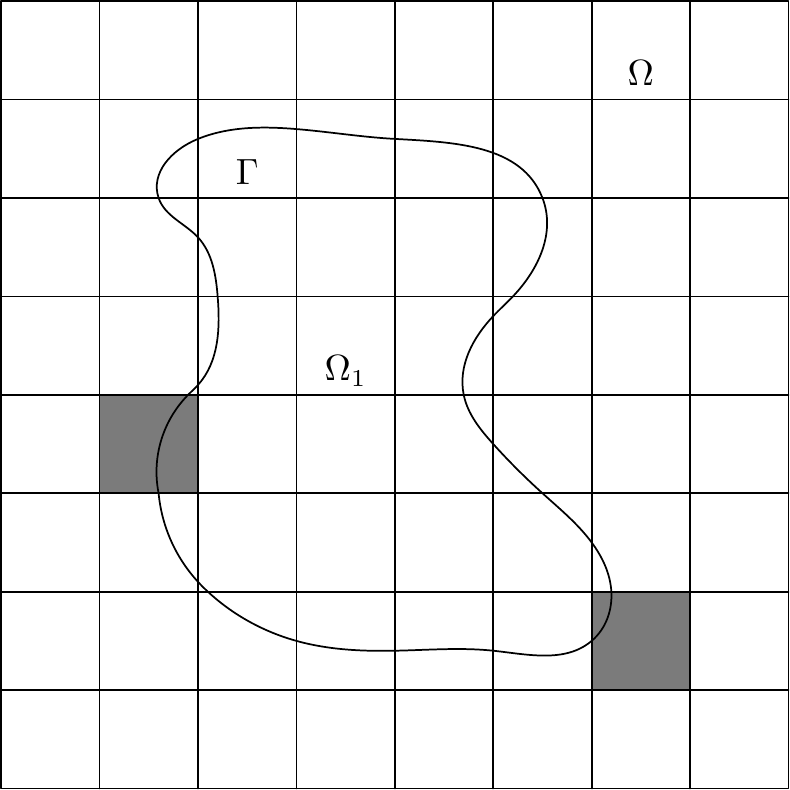}
   \caption{The setting of the elliptic interface problem and the unfitted mesh.}\label{fig:1.1}
\end{figure}

Unfitted finite element methods in the discontinuous Galerkin (DG) framework have attracted considerable interests in the literature in the last twenty years starting from the seminal work \cite{Hansbo} in which an unfitted finite element method is proposed for elliptic interface problems. The method is defined on a fixed background mesh and uses different finite element functions in different cut cells which is the intersection of the mesh elements and physical domains. The jump condition on the interface is enforced by penalties which extends an earlier idea of Nitsche \cite{Nitsche}. This unfitted finite element method can also be viewed as the interior penalty discontinuous Galerkin method (see, e.g., \cite{Arnold}) defined on meshes allowing curve-shaped elements. The main difficulty in using the unfitted finite element methods is the so-called {\it small cut cell problem}: the cut cells can be arbitrarily small and anisotropic, which can make the stiffness matrix extremely ill-conditioned, especially for high-order finite element methods \cite{Prenter, Badia22}. For other approaches to design unfitted discretization methods by constructing special finite element bases on interface elements or finite difference stencils along the interface, we refer to the immersed boundary method \cite{Peskin}, the immersed interface method \cite{LeVeque, Li06}, or the immersed finite element method \cite{Li03, Chen09}.

There are two approaches in the literature to attack the small cut cell problem. One is by appropriate techniques of stabilization \cite{Burman10, Burman12, Wu, Massjung, Wang, Gurken}. Among them, for example, the method of ghost penalty \cite{Burman10, Burman12, Gurken} adds additional penalties on the jumps of derivatives across sides or facets of interface elements. The other approach is by merging the small cut cells with neighboring large elements \cite{Johansson, Huang, Badia18, CLX, BurmanHHO21} so that the merged macro-elements have enough support. While the DG formulation is still used in \cite{Johansson, Huang, CLX},  the aggregated unfitted finite element method in \cite{Badia18} relies on the construction of stable extension operators so that the finite element space is still $C^0$. We refer to recent works \cite{Burman15, Burman21, Badia21, Badia22} for further information about ghost penalty and the aggregated unfitted finite element method.

In \cite{CLX} an adaptive high-order unfitted finite element method is proposed for elliptic interface problems in which the $hp$ a priori and a posteriori error estimates are derived based on novel $hp$ domain inverse estimates and the concept of interface deviation. The interface deviation is a measure to quantify the mesh resolution of the geometry of the interface. We remark that the study on $hp$ inverse estimates on curved domains is not only of mathematical interests, it is also essential to understand and control the exponential growth on the finite element approximation order $p$ of the condition number of the stiffness matrix of the unfitted finite element method in this paper.

The macro-elements, which are the union of small interface elements and their surrounding elements, are assumed to be rectangular in \cite{CLX}. This assumption is different from those in \cite{Johansson, Huang, Badia18}, see Fig.\ref{fig:p123}. The macro-elements in \cite{Johansson, Huang, Badia18} need not to be of rectangular shape, which makes the implementation simpler but the crucial inverse estimates on extended elements in \cite{Johansson, Huang} or the stability of the extension operators \cite{Badia18} are shown without considering the dependence on the finite element approximation order $p$. The assumption that the macro-elements should be rectangular in \cite{CLX} raises the question of how to construct the merging algorithm in practical applications.

\bigskip
\begin{figure}[h]
  \centering
  \includegraphics[width=0.75\linewidth]{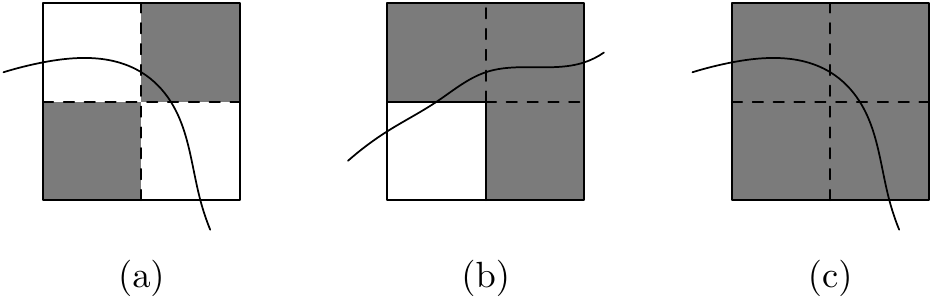}
   \caption{Three different ways of generating macro-elements which are marked in dark. The left, middle, and right figures illustrate the macro-elements used in \cite{Johansson,BurmanHHO21}, \cite{Huang, Badia18}, and \cite{CLX}, respectively.}\label{fig:p123}
\end{figure}

The first objective of this paper is to propose a reliable algorithm to merge small interface elements with their surrounding elements to generate the macro-elements. The algorithm is based on the concept of admissible chain of interface elements, the classification of patterns for merging elements, and appropriate ordering in generating macro-elements from the patterns so that the reliability of the algorithm in the sense that it terminates in finite number of steps can be proved. This algorithm also leads to a reliable algorithm of automatically generating 2D shape regular body-fitted finite element meshes for arbitrarily shaped smooth interfaces. To the authors' best knowledge, this algorithm introduces a new way to generate body-fitted finite element meshes and may be of independent interest.

The second objective of the paper is to study the condition number of the stiffness matrix of high-order unfitted finite element methods which are known to be of the order $O(h^{-2})$ in the literature \cite{Burman10, Johansson, Huang, Badia18, Badia22HHO} on quasi-uniform meshes with the mesh size $h$.
For high order methods, it is known \cite{Prenter} that the condition number of the stiffness matrix may grow exponentially with the finite element approximation order $p$ in terms of the measure of cut cells.  This indicates that the geometry of the cut cells is essential in controlling the condition number of the stiffness matrix.

In this paper, we will take the basis functions of the spectral element, that is, the Lagrangian interpolation functions at the Gauss-Lobatto points on elements not intersecting with the interface. For the interface elements, extra care must be taken as the basis of the spectral element on $K$ is ill-conditioned on the subsets $K_i=K\cap\Om_i$, $i=1,2$, which is similar to the observation in \cite[P.346]{Dubiner} for Legendre polynomials. Here we choose the $L^2$-orthogonal functions on some special polygons inside $K_i$, $i=1,2$, as the basis functions for the interface elements $K$. We show that the condition number of the stiffness matrix is bounded by $\Theta^2(p^3(N-N^\Gamma)+p^4N^\Gamma)$ up to a logarithmic factor, where $N$ is the number of total elements, $N^\Ga$ is the number of interface elements, and $\Theta$ depends on the interface deviation and $p$. This bound is optimal and indicates that the mesh has to sufficiently resolve the geometry of the interface to control the condition number of the stiffness matrix.

The results of this paper allow for extensions in several directions. Firstly, for the ease of exposition, we consider in this paper the case when the domain $\Omega$ is a union of rectangles and the interface is smooth. The extension to the general domains with smooth boundary is straightforward. Secondly, the case when the interface is piecewise smooth will be pursued in our forthcoming work by combining the ideas in \cite{CLX} on large elements and interface deviation for interfaces with singularities with the merging algorithm developed in this paper. Thirdly, the theoretical results in this paper and in \cite{CLX} including the $hp$ domain inverse estimates and the concept of the interface deviation can be extended to study three-dimensional interface problems. The merging algorithm in the three-dimensional case is more challenging. Nevertheless, we believe that with the new insights gained in this paper for the two-dimensional case, reliable algorithms for constructing cubic macro-elements can be achieved in future. Finally, we remark that our argument to analyze and control the condition number of the stiffness matrix is fairly general, it can be used in other unfitted finite element methods including three-dimensional cases.

The layout of the paper is as follows. In section 2 we introduce our unfitted finite element method. In section 3 we construct the merging algorithm to generate the induced mesh. In section 4 we prove the discrete Poincar\'e inequality and the $hp$ estimate for the condition number of the stiffness matrix. In section 5 we present several numerical examples to confirm our theoretical results.

\section{The unfitted finite element method}\label{sec_ufem}
Let $\Om\subset\R^2$ be a domain which is a union of rectangles and $\cT$ a Cartesian finite element mesh of $\Om$
with possible hanging nodes. This allows us to locally refine the mesh near the interface to resolve the geometry to save the computational costs away from the interface. The elements of the mesh are (open) rectangles whose sides are parallel to the coordinate axes. We assume that the interface intersects the boundary of $K$ twice at different sides (including the end points).

For any element $K$, let $h_K$ stand for its diameter. Denote $\cT^\Gamma:=\{K\in\cT:K\cap\Gamma\not=\emptyset\}$ the set of interface elements. We recall the definition of large element in Chen et al \cite[Definition 2.1]{CLX}.

\begin{Def}\label{def:2.1}
(Large element) For $i=1,2$, an element $K\in \caT$ is called a large element with respect to $\Omega_i$ if $K\subset\Omega_i$ or $K\in\caT^\Gamma$ for which there exists a constant $\delta_0\in(0,1/2)$ such that $|e\cap \Omega_i|\ge \delta_0|e|$ for each side $e$ of $K$ having nonempty intersection with $\Omega_i$. Specially, $K$ is called a large element if $K\in\caT^\Gamma$ is large with respect to both $\Omega_1$ and $\Omega_2$. Otherwise, $K$ is called a small element.
\end{Def}

Note that it is possible that $K\in\caT^\Gamma$ may not be a large element. The following assumption in \cite{CLX} is inspired by Johansson and Larson \cite{Johansson} in which a fictitious boundary method is considered.

\medskip
{\bf Assumption (H1)} For each $K\in\cT^\Gamma$, there exists a rectangular macro-element $M(K)$ which is a union of $K$ and its surrounding element (or elements) such that $M(K)$ is a large element. We assume $h_{M(K)}\le C_0h_K$ for some constant {\rev{$C_0>0$}}.
\medskip

In section 3 we will construct a merging algorithm to find the macro-element for each small element in an admissible chain of interface elements. This indicates that the assumption (H1) can always be satisfied by using the algorithm. In the following, we will always set $M(K)=K$ if $K\in\caT^{\Gamma}$ is a large element. Then, the induced mesh of
$\caT$ is defined as
\ben
\cM=\{M(K):K\in \mathcal{T}^\Gamma\}\cup\{K\in\mathcal{T}: K\not \subset M(K') \text{ for some } K'\in \mathcal{T}^{\Gamma}\}.
\een
We will write $\cM={\rm Induced}(\cT)$. Note that $\mathcal{M}$ is also a Cartesian mesh of $\Omega$ in the sense that either $M(K)\cap M(K')=\emptyset$ or $M(K)=M(K')$ for any two different elements $K,K'\in\mathcal{T}$. All elements in $\mathcal{M}$ are large elements.

For any $K\in\mathcal{M}^\Gamma:=\{K\in\mathcal{M}:K\cap\Gamma\not=\emptyset\}$, denote $K_i=K\cap\Om_i$, $i=1,2$, $\Ga_K=\Ga\cap K$, and $\Gamma_K^h$ the open line segment connecting the two intersection points of $\Gamma$ and $\pa K$. $\Ga_K^h$ divides the element $K$ into two polygons $K_1^h$ and $K_2^h$ which are the polygonal approximation of $K_1$ and $K_2$, respectively. An important property of $K$ being a large element is that $K_i^h$, $i=1,2$, is a {\it strongly} shape regular polygon in the sense that it is the union of shape regular triangles in the sense of Ciarlet \cite{Ciarlet}. We remark that there are different definitions of shape regular polygons in the literature, see, e.g., Ming and Shi \cite{Ming} and Brenner and Sung \cite{Brenner}.

The following concept of interface deviation is introduced in \cite{CLX}.

\begin{Def}\label{def:2.2}
For any $K\in\mathcal{M}^\Gamma$, the interface deviation $\eta_K$ is defined as $\eta_K=\max(\eta_K^1,\eta_K^2)$, where for $i=1,2$, if $A_K^i\in\Om_i$ is the vertex of $K$ which has the maximum distance to $\Gamma_K^h$ among all vertices of $K$ in $\Om_i$,
\ben
\eta_K^i=\frac{\dist_{\rm H}(\Gamma_K,\Gamma_K^h)}{\dist(A_K^i,\Gamma_K^h)}.
\een
Here $\dist_{\rm H}(\Gamma_1,\Gamma_2)=\max_{x\in\Gamma_1}(\min_{y\in\Gamma_2}|x-y|)$ and $\dist(A,\Gamma_1)=\min_{y\in\Ga_1}|A-y|$.
\end{Def}

The interface deviation is a measure on how well the mesh resolves the geometry of the interface. We will show in section 4 that this concept also links to the control of the condition number of the stiffness matrix.

It is known that if $\Gamma_K$ is $C^2$-smooth, $\dist_{\rm H}(\Gamma_K,\Gamma_K^h)\le Ch_K^2$ (see, e.g., Feistauer \cite[\S3.3.2]{Feistauer}) and thus $\eta_K\le Ch_K$ for some constant $C$ independent of $h_K$. Therefore, the interface deviation can be made arbitrarily small by locally refining the mesh near the interface.
When the interface $\Gamma$ is Lipschitz and piecewise $C^2$-smooth, the definition of the large element and interface deviation has to be modified in the elements containing the singular points of the interface, see \cite{CLX} for the details.

For any integer $p\ge 1$ and $K\in\mathcal{M}$, denote $Q_p(K)$ the set of polynomials in $K$ which is of degree $p$ in each variable. The following $hp$ domain inverse estimate is proved in \cite[Lemma 2.4]{CLX}.

\begin{lem}\label{lem:2.0}
Let $\Delta$ be a triangle with vertices $A=(a_1,a_2)^T$, $B=(0,0)^T$, $C=(c_1,0)^T$, where $a_2,c_1 >0$. Let $\delta\in (0,a_2)$ and $\Delta_\delta=\{x\in\Delta:{\rm dist}(x,BC)>\delta\}$. Then we have
\ben
\|v\|_{L^2(\Delta)}\le\mathsf{T}\left(\frac{1+\delta a_2^{-1}}{{\rev{1-\delta a_2^{-1}}}}\right)^{2p+3/2}\|v\|_{L^2(\Delta_\delta)}\ \ \forall v\in Q_p(\Delta),
\een
where $\mathsf{T}(t)=t+\sqrt{t^2-1}\ \ \forall t\ge 1$.
\end{lem}

The proof of this lemma makes use of the following one-dimensional domain inverse estimate in \cite[Lemma 2.3]{CLX}
\beq\label{g1}
\|g\|_{L^2(I_\lambda\backslash\bar I)}^2\le \frac 12\left[(\lambda+\sqrt{\lambda^2-1})^{2p+1}-1\right]\|g\|_{L^2(I)}^2\ \ \forall g\in Q_p(I_\lambda),
\eeq
where $I=(-1,1), I_\lambda=(-\lambda,\lambda)$, $\lambda>1$, and $Q_p(I_\lambda)$ is the set of polynomials of order $p$ in $I_\lambda$. We remark that the growing factor $(\lambda+\sqrt{\lambda^2-1})^{2p+1}$ is sharp which is attained by the Chebyshev polynomials whose explicit expression is $C_n(t)=\frac 12[(t+\sqrt{t^2-1})^n+(t-\sqrt{t^2-1})^n]$, $n\ge 0$, see DeVore and Lorentz \cite[P.76]{DeVore}.

Let $\delta_K:=\dist_{\rm H}(\Gamma_K,\Gamma_K^h)$, We also define two polygons $K_i^{h-\delta_K}$, $i=1,2$, as follows. Let $\Gamma_{K_i}^{h-\delta_K}\subset K_i$ be the line segment which is parallel to $\Gamma_K^h$ and its distance to $\Ga^h_K$ is $\delta_K$. Let $K_i^{h-\delta_K}$ be the polygon bounded by sides of $K$ and $\Ga_{K_i}^{h-\delta_K}$.

\begin{figure}[h]
\centering
\includegraphics[width=0.35\textwidth]{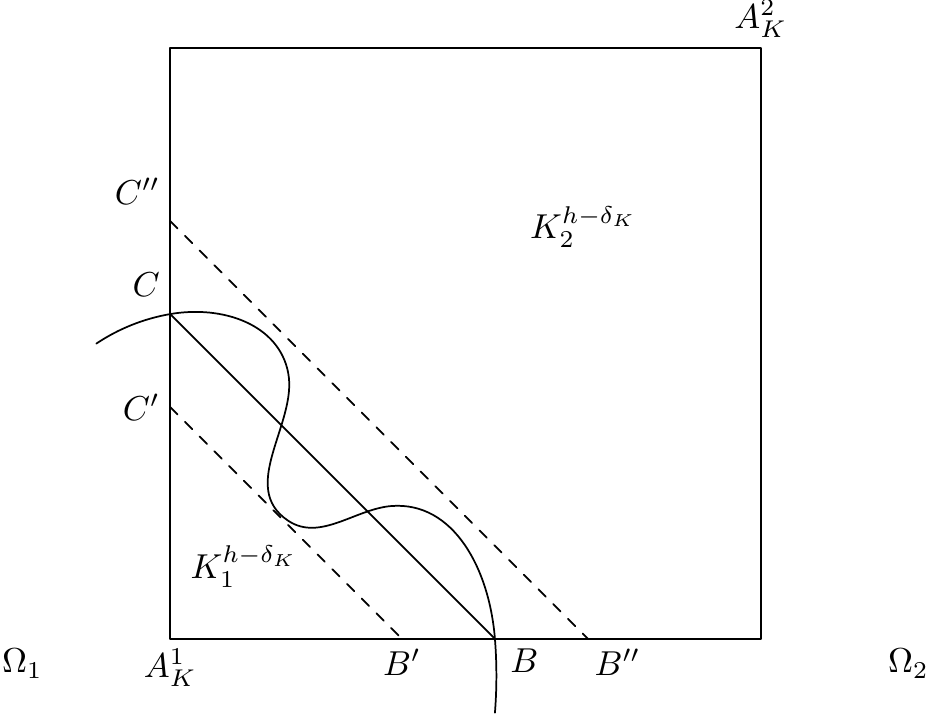}
\caption{The figure used in the proof of Lemma \ref{lem:new}.}\label{fig:2.1}
\end{figure}

\begin{lem}\label{lem:new}
Let $K\in\cM^\Ga$ and $\eta_K\le 1/2$, Then for $i=1,2$, we have
\begin{align}
& \|v_i\|_{L^2(K_i^{h-\delta_K})}\le \|v_i\|_{L^2(K_i)}\le C\mathsf{T}\left(\frac{1+3\eta_K}{1-\eta_K}\right)^{2p+3/2}\|v_i\|_{L^2(K_i^{h-\delta_K})}\ \ \forall v\in Q_p(K),\label{f4}
\end{align}
where the constant $C$ is independent of $h_K$, $p$, and $\eta_K$.
\end{lem}

\begin{proof} The left inequality \eqref{f4} is trivial since $K_i^{h-\delta_K}\subset K_i$. Here we prove the right inequality in \eqref{f4} when $\Ga$ intersects $\pa K$ at neighboring sides. The other cases can be proved similarly.

We use the notation in Fig.\ref{fig:2.1} in which $B'C'$, $B''C''$ are parallel to $\Ga_K^h$ and the distances of $B'C', B''C''$ to $\Ga_K^h$ are $\delta_K$. Then $K_1^{h-\delta_K}=\Delta A_K^1B'C'$ and $K_2^{h-\delta_K}$ is the polygon bounded by sides of $K$ and $B''C''$. Let $d_i=\dist(A_K^i,\Ga_K^h)$, $i=1,2$. By definition, the interface deviation $\eta_K\ge \delta_K/d_i$, $i=1,2$. By Lemma \ref{lem:2.0}, for any $v\in Q_p(K)$,
\ben
\|v\|_{L^2(K_1)}\le\|v\|_{L^2(\Delta A_K^1B''C'')}&\le&\mathsf{T}\left(\frac{1+2\delta/(d_1+\delta)}{1-2\delta/(d_1+\delta)}\right)^{2p+3/2}\|v\|_{L^2(\Delta A_K^1B'C')}\\
&\leq&\mathsf{T}\left(\frac{1+3\eta_K}{1-\eta_K}\right)^{2p+3/2}\|v\|_{L^2(K_1^{h-\delta_K})}.
\een
The case for $K_2$ can be proved similarly. This completes the proof.
\end{proof}

The numerical results in Example 1 in section 5 indicate that the bound in Lemma \ref{lem:new} is sharp. Now for any $K\in\mathcal{M}$, we denote
\ben
a_K=\left\{\begin{array}{ll}
\frac{a_1+a_2}{2}& \mbox{if }K\in\mathcal{M}^\Ga,\\
a_i & \mbox{if }K\in \Omega_i.
\end{array}\right.,\quad
\Theta_K=\left\{\begin{array}{ll}
\mathsf{T}\left(\frac{1+3\eta_K}{1-\eta_K}\right)^{4p+3} & \mbox{if }K\in\mathcal{M}^\Ga,\\
1 & \mbox{otherwise}.
\end{array}\right.
\een

Based on the concept of interface deviation, the following $hp$ inverse estimates on curved domains are proved in \cite[Lemma 2.8, (2.12)]{CLX}.

\begin{lem}\label{lem:2.1}
Let $K\in\cM^\Ga$ and $\eta_K\le 1/2$, Then for $i=1,2$, we have
\ben
& &\|\na v\|_{L^2(K_i)}\le Cp^2h_K^{-1}\Theta_K^{1/2}\|v\|_{L^2(K_i)}\ \ \forall v\in Q_p(K),\\
& &\|v\|_{L^2(\pa K_i)}\le Cp h_K^{-1/2}\Theta_K^{1/2}\|v\|_{L^2(K_i)}\  \ \forall v\in Q_p(K),
\een
where the constant $C$ is independent of $h_K,p,$ and $\eta_K$.
\end{lem}

We remark that $hp$ inverse estimates on star-shaped curve elements are studied in Massjung \cite{Massjung}, Wu and Xiao \cite{Wu}, and Cangiani et al \cite{Dong} which can be viewed as different forms of assumption on the mesh to resolve the geometry.
 Lemma \ref{lem:2.1} does not require the locally star-shaped assumption on the interface and is robust with respect to small variations of the interface as long as the interface deviation is the same.

Notice that if $\eta_K\le \frac 1{p(p+1)}$, for $s=\frac{1+3\eta_K}{1-\eta_K}=1+\gamma_K$, where $\gamma_K=\frac {4\eta_K}{1-\eta_K}\le 4p^{-2}$, we have
$\mathsf{T}(s)=s+\sqrt{s^2-1}=1+\rho_K$ with $\rho_K=\gamma_K+\sqrt{\gamma_K^2+2\gamma_K}\le p^{-1}(4p^{-1}+\sqrt{16p^{-2}+8})$.
Thus $\Theta_K=e^{(4p+3)\ln(\mathsf{T}(s))}\le e^{(4p+3)\rho_K}\le C$ for some constant $C$ independent of $p$ and $\eta_K$. This motivates us to make the following assumption in the remainder of this paper which can be easily satisfied for $C^2$-smooth interfaces if the mesh is locally refined near the interface.

\medskip
{\bf Assumption (H2)} For any $K\in\mathcal{M}^\Gamma$, $\eta_K\le \frac 1{p(p+1)}$.
\medskip

Now we introduce some notation for DG methods. Let $\cE=\cE^{\rm side}\cup\cE^\Ga\cup\cE^{\rm bdy}$, where $\cE^{\rm side}=\{e=\pa K\cap\pa K':K,K'\in\cM\}$, $\cE^\Ga=\{\Ga_K:K\in\cM\}$, and $\cE^{\rm bdy}=\{e=\pa K\cap\pa\Om:K\in\cM\}$. For $i=1,2$, denote by $\cM_i=\{K\in\cM:K\cap\Om_i\not=\emptyset\}$. Then $\Om_i\subset\Om_i^h=\cup\{K:K\in\cM_i\}$. We denote $\cE_i^{\rm side}$ the set of all sides of $\cM_i$ interior to $\Om_i^h$, that is, not on the boundary $\pa\Om_i^h$. Finally, we set $\bar{\cal{E}}= \cE^{\rm side}_{1} \cup\cE^{\rm side}_{2}\cup\cE^\Ga\cup\cE^{\rm bdy}$.

For any $e\in\cE$, we fix a unit normal vector $n_e$ of $e$ with the convention that $n_e$ is the unit outer normal to $\pa\Om$ if $e\in\cE^{\rm bdy}$ and $n_e$ is the unit outer normal to $\pa\Om_1$ if $e\in\cE^\Ga$. For any $v\in H^1(\cM):=\{v_1\chi_{\Om_1}+v_2\chi_{\Om_2}:v_i|_K\in H^1(K), K\in\cM, i=1,2\}$, we define the jump of $v$ across $e$ as
\ben
\lj v\rj_e:=v_--v_+\ \ \forall e\in\cE^{\rm side}\cup\cE^\Ga,\ \ \ \
\lj v\rj_e:=v_-\ \ \forall e\in\cE^{\rm bdy},
\een
where $v_\pm$ is the trace of $v$ on $e$ in the $\pm n_e$ direction. We define the normal vector function $n\in L^\infty(\cE)$ by
$n|_e=n_e\ \ \forall e\in\cE$.

For any subset $\widehat\cM\subset\cM$ and $\hat\cE\subset\bar\cE$, we use the notation
\ben
(u,v)_{\widehat\cM}:=\sum_{K\in\widehat\cM}(u,v)_K,\ \ \la u,v\ra_{\hat\cE}:=\sum_{e\subset\hat\cE}\la u,v\ra_e,
\een
where $(u,v)_K$ is the inner product of $L^2(K)$ and $\la u,v\ra_e$ is the inner product of $L^2(e)$.

The unfitted finite element method is based on the idea of ``doubling of unknowns" in Hansbo and Hansbo \cite{Hansbo}. We define
the unfitted finite element space as
\ben
\mathbb{X}_p(\cM)=\{v_1\chi_{\Om_1}+v_2\chi_{\Om_2}:v_i|_K\in Q_p(K), K\in\cM,i=1,2\}.
\een
For any $v\in H^1(\cM)$, we denote $\na_hv|_K:=\na v_1\chi_{K_1}+\na v_2\chi_{K_2}$, where $\chi_{K_i}$ is the characteristic function of $K_i$, $i=1,2$. For any $v\in H^1(\cM),g\in L^2(\pa\Om)$, we define the liftings $\mathsf{L}(v)\in [\X_p(\cM)]^2$, $\mathsf{L}_1(g)\in [\X_p(\cM)]^2$ such that
\beq
(w,\mathsf{L}(v))_\cM=\la w^-\cdot n,\lj v\rj\ra_{\cE},\ \ \ \
(w,\mathsf{L}_1(g))_\cM=\la w\cdot n,g\ra_{\cE^{\rm bdy}}\ \ \ \forall w\in [\X_p(\cM)]^2\label{ll1}
\eeq
Our unfitted finite element method is to find $U\in\X_p(\cM)$ such that
\beq\label{a2}
a_h(U,v)=F_h(v)\ \ \ \ \forall v\in\X_p(\cM),
\eeq
where the bilinear form $a_h: H^1(\cM)\times H^1(\cM)\to \R$, and the functional $F_h:H^1(\cM)\to\R$ are
given by
\begin{align}
a_h(v,w)=&(a(\na_h v-\mathsf{L}(v)),\na_h w-\mathsf{L}(w))_\cM+\la\al\lj v\rj,\lj w\rj\ra_{\bar{\cal{E}}}
+\la p^{-2}h\na_T\lj v\rj,\na_T\lj w\rj\ra_{\cE^\Ga},\label{a3}\\
F_h(v)=&(f,v)_\cM-(a\mathsf{L}_1(g),\na_h v-\mathsf{L}(v))_\cM+\la\al g,v\ra_{\cE^{\rm bdy}},\label{a33}
\end{align}
where $\na_T$ is the surface gradient on $\Ga$. For any $v=v_1\chi_{\Om_1}+v_2\chi_{\Om_2}, w=w_1\chi_{\Om_1}+w_2\chi_{\Om_2}\in H^1(\cM)$,
\ben
\la\al\lj v\rj,\lj w\rj\ra_{\bar{\cal{E}}}:=\sum^2_{i=1}\la\al\lj v_i\rj,\lj w_i\rj\ra_{\cE_i^{\rm side}}+\la\al \lj v\rj,\lj w\rj\ra_{\cE^\Ga\cup\cE^{\rm bdy}}.
\een
The interface penalty function $\al\in L^\infty(\cE)$ is
\beq
\al |_e=\al_0 a_e\Theta_eh_e^{-1}p^2\ \ \forall e\in\cE,\label{a34}
\eeq
where $\al_0>0$ is a fixed constant, $a_e=\max\{a_K:e\cap \bar K\not=\emptyset\}\ \forall e\in\cE$, $\Theta_e=\max\{\Theta_K:e\cap \bar K\not=\emptyset\}\ \forall e\in\cE$, and the mesh function $h|_e=(h_K+h_{K'})/2$ if $e=\pa K\cap\pa K'\in\cE^{\rm side}$ and $h|_e=h_K$ if $e=K\cap\Ga\in\cE^\Ga$ or $e=\pa K\cap\pa\Om\in\cE^{\rm bdy}$.

We remark that our unfitted finite element method \eqref{a2} is the so-called local discontinuous Galerkin (LDG) method in Cockburn and Shu \cite{Cockburn} which is different from the interior penalty discontinuous Galerkin (IPDG) method used in \cite{Hansbo}. We choose the LDG method because the penalty constant $\al_0$ in \eqref{a34} can be any fixed constant, while the corresponding penalty constant in the IPDG method has to be sufficiently large to ensure the stability. We refer to Arnold et al \cite{Arnold} for a review of different DG methods for elliptic equations.

Notice that the last term $\la p^{-2}h\na_T\lj v\rj, \na_T\lj w\rj\ra_{\cE^\Ga}$ in the bilinear form \eqref{a3} is not present in \cite{CLX}. It is included in this paper in order to show the discrete Poincar\'e inequality for unfitted finite element functions in Lemma \ref{lem:2.3} which is crucial for us to study the condition number of the stiffness matrix. We also remark that $\la p^{-2}h\na_T\lj v\rj, \na_T\lj w\rj\ra_{\cE^\Ga}$ penalizes the tangential gradient of the finite element solution, not the normal flux of the solution as in Burman and Hansbo \cite{Burman10}, Xiao and Wu \cite{Wu}.

For any $v\in H^2(\cM)$, we introduce the DG norm
\ben
\|v\|_{\rm DG}^2:=\|a^{1/2}\na v\|_\cM^2+\|\al^{1/2}\lj v\rj\|_{\bar{\cE}}^2+\|p^{-1}h^{1/2}\na_T\lj v\rj\|_{\cE^\Ga}^2,
\een
where $\|\al^{1/2}\lj v\rj \|_{\bar{\cE}}^2=\la\al\lj v\rj,\lj v\rj\ra_{\bar{\cE}}$ and $\|p^{-1}h^{1/2}\na_T\lj v\rj\|_{\cE^\Ga}^2=\la p^{-2}h\na_T\lj v\rj,\na_T\lj v\rj\ra_{\cE^\Ga}$. By Lemma \ref{lem:2.1}, it is easy to show that
\ben
a_h(v,v)\le C\|v\|_{\rm DG}^2\ \ \ \forall v\in X_p(\cM).
\een
Moreover, by \cite[Theorem 2.1]{CLX} we know that
\ben
a_h(v,v)\ge c_{\rm stab}\|v\|_{\rm DG}^2\ \ \ \forall v\in\X_p(\cM),
\een
where $c_{\rm stab}>0$ is a constant independent of the mesh sizes, $p$, and the interface deviations $\eta_K$ for all $K\in\cM^\Ga$.

\begin{thm}\label{thm:2.1}
Let the solution of the problem \eqref{m1}-\eqref{m2} $u\in H^k(\Om_1\cup\Om_2)$, $k\ge 2$, and $U\in\X_p(\cM)$ be the solution of the problem \eqref{a2}. Then we have
\ben
\|u-U\|_{\rm DG}\le C\Theta^{1/2}\frac{h^{\min(p+1,k)-1}}{p^{k-3/2}}\|u\|_{H^k(\Om_1\cup\Om_2)},
\een
where $h=\max_{K\in\cM}h_K$, $\Theta=\max_{K\in\cM}\Theta_K$, and the constant $C$ is independent of the mesh sizes, $p$, and the interface deviations $\eta_K$ for all $K\in\cM^\Ga$.
\end{thm}

\begin{proof} For the sake of completeness, we sketch a proof by using the argument in e.g., Perugia and Sch\"otzau \cite{Perugia}, Wu and Xiao \cite{Wu}. For $i=1,2$, let $\tilde u_i\in H^k(\R^2)$ be the Stein extension (cf., e.g., Adams and Fournier \cite[Theorem 5.14]{Adams}) of $u_i=u|_{\Om_i}\in H^k(\Om_i)$, which is available for any Lipschitz domains, such that $\|\tilde u_i\|_{H^k(\R^2)}\le C\|u_i\|_{H^k(\Om_i)}$. Let $u_I=I_{hp}(\tilde u_1)\chi_{\Om_1}+I_{hp}(\tilde u_2)\chi_{\Om_2}$, where $I_{hp}:H^1(\cM)\to\V_p(\cM)=\Pi_{K\in\cM}Q_p(K)$ is the interpolation operator defined in Babu\v{s}ka and Suri \cite[Lemma 4.5]{Babuska87b}. For any $K\in\cM$, it satisfies that for any $0\le j\le k$,
\beq\label{a8}
\|w-I_{hp}(w)\|_{H^j(K)}\le C\frac{h_K^{\min(p+1,k)-j}}{p^{k-j}}\|v\|_{H^k(K)}\ \ \forall v\in H^k(K),
\eeq
where the constant $C$ is independent of $h_K,p$, but may depend on $k$.
By the multiplicative trace inequality, we have
\ben
\|w\|_{L^2(\pa K)}\le Ch_K^{-1/2}\|w\|_{L^2(K)}+C\|w\|_{L^2(K)}^{1/2}\|\na w\|_{L^2(K)}^{1/2}\ \ \forall w\in H^1(K).
\een
For any $K\in\cM^\Ga$, by Xiao et al \cite[Lemma 3.1]{Wang}, \cite[Lemma 2.6]{CLX}, we have that for $i=1,2$,
\beq\label{a10}
\|w\|_{L^2(\Ga_K)}\le C\|w\|_{L^2(K_i)}^{1/2}\|\na w\|_{L^2(K_i)}^{1/2}+\|w\|_{L^2({\rev{\pa K_i\backslash\bar\Ga_K}})}\ \ \forall w\in H^1(K).
\eeq
Thus we obtain by using \eqref{a8} that for any $K\in\cM$, $j=0,1$,
\ben
\|w-I_{hp}(w)\|_{H^j(\pa K_i)}\le C\frac{h^{\min(p+1,k)-j-1/2}}{p^{k-j-1/2}}\|w\|_{H^k(K)}\ \ \forall w\in H^k(K).
\een
This implies easily that
\beq\label{a9}
\|u-u_I\|_{\rm DG}\le C\Theta^{1/2}\frac{h^{\min(p+1,k)-1}}{p^{k-3/2}}\|u\|_{H^k(\Om_1\cup\Om_2)}.
\eeq
On the other hand, since $a_h(u,v)=F_h(v)\ \ \forall v\in\X_p(\cM)$, we use \eqref{a2} to conclude that
\ben
\|u_I-U\|_{\rm DG}^2\le c_{\rm stab}^{-1}a_h(u_I-U,u_I-U)&=&c_{\rm stab}^{-1}a_h(u_I-u,u_I-U)\\
&\le&C\|u_I-u\|_{\rm DG}\|u_I-U\|_{\rm DG}.
\een
This completes the proof by \eqref{a9} and the triangle inequality.
\end{proof}

To conclude this section, we remark that the same a posteriori error estimate in \cite[Theorem 3.1]{CLX} also holds for the solution $U\in\X_p(\cM)$ in \eqref{a2}. Here we omit the details.

\section{The merging algorithm}\label{sec_algo_1}

In this section, we construct a merging algorithm for the admissible chain  of interface elements so that each small interface element in the chain is included in some macro-element which is a large element. We first introduce the concept of admissible chain in \S 3.1 and five types of patterns of merging small interface elements with their surrounding elements in \S 3.2. We propose our merging algorithm and prove its reliability in \S 3.3.

\subsection{The admissible chain of interface elements}

A chain of interface elements $\mathfrak{C}=\{G_1\rightarrow G_2 \rightarrow\cdots \rightarrow G_n\}$ orderly consists of $n$ interface elements $G_i\in\cT^\Gamma$, $i=1,\cdots,n$, such that $\bar\Gamma_{G_i}\cup\bar\Gamma_{G_{i+1}}$ is a continuous curve, $1\le i\le n-1$. We call $n$ the length of $\mathfrak{C}$ and denote $\mathfrak{C}\{i\}=G_i$, $i=1,\cdots,n$.

For any element $K\in\cT$, we call $N(K)\in\cT$ a neighboring element of $K$ if $K$ and $N(K)$ share a common side, and $D(K)\in\cT$ a diagonal element of $K$ if $K$ and $D(K)$ only share one common vertex. Set $\cS(K)_0=\{K\}$, and for $j\ge 1$, denote $\cS(K)_j=\{K''\in\cT:\exists\,K'\in\cS(K)_{j-1}\ \mbox{such that }\bar K''\cap\bar K'\not=\emptyset\}$, that is, $\cS(K)_j$ is the set of all $k$-th layer
elements surrounding $K$, $0\le k\le j$. Obviously, $\cS(K)_0\subset\cS(K)_1\subset\cdots\subset\cS(K)_j$ for any $ j\ge 1$.

\begin{Def}\label{def:3.1}
A chain of interface elements $\mathfrak{C}$ is called admissible if the following rules are satisfied.
\begin{description}
\item[$1.$] For any $K\in\mathfrak{C}$, all elements in $\cS(K)_2$ have the same size as
that of $K$.
\item[$2.$] If $K\in\mathfrak{C}$ has a side $e$ such that $\bar e\subset \Om_i$, then $e$ must be a side of some neighboring element $N(K)\subset\Om_i$, $i=1,2$.
\item[$3.$] Any elements $K\in\cT\backslash\cT^\Gamma$ can be neighboring at most two elements in $\mathfrak{C}$.
\item[$4.$] For any $K\subset\Om_i$, the interface elements in {$\cS(K)_j$, $j=1,2$,} must be connected in the sense that the interior of the closed
set $\cup\{\bar G: G\in\cS(K)_{j}\cap{\revto{\cT}}^\Ga\}$ is a connected domain.
\end{description}
\end{Def}

\bigskip
\begin{figure}[h]
\centering
\includegraphics[width=0.6\textwidth]{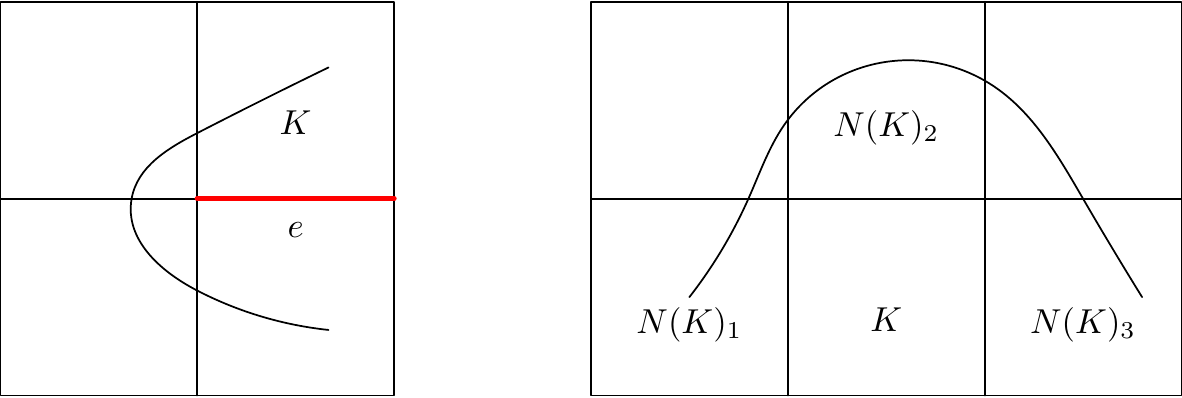}
\caption{The patch of elements not allowed by Rule $2$ (left) and Rule $3$ (right) in Definition \ref{def:3.1}. }\label{fig:3.2}
\end{figure}

\begin{figure}[h]
\centering
\includegraphics[width=0.8\textwidth]{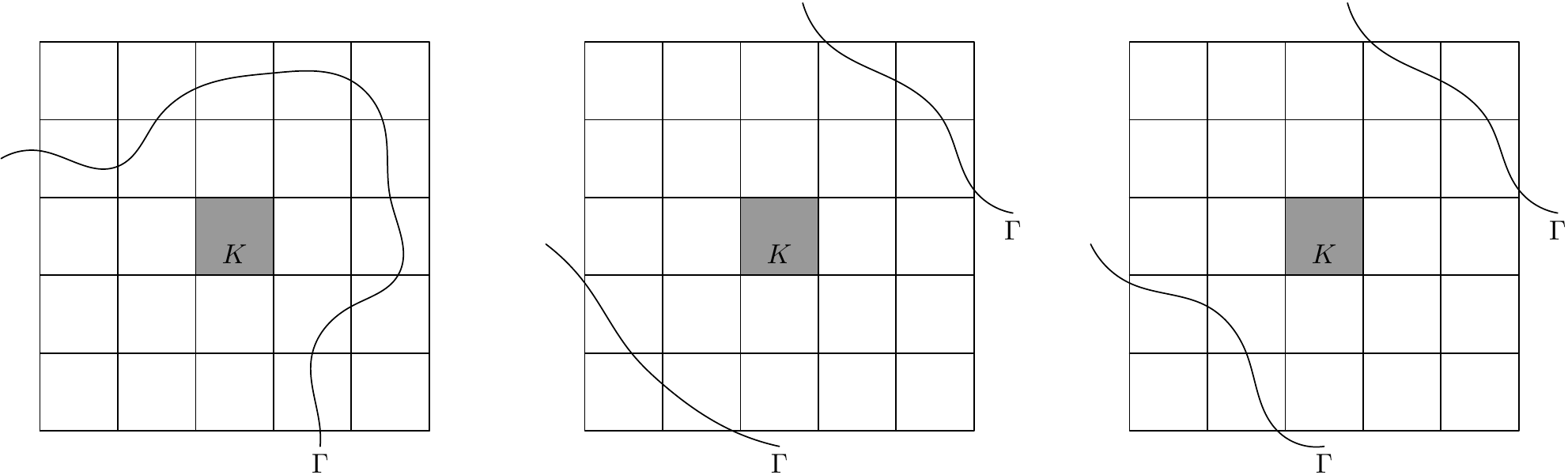}
\caption{The patch of elements not allowed by Rule $4$ in Definition \ref{def:3.1}. }\label{fig:3.3}
\end{figure}

We remark that the four rules of the admissible chains can be easily satisfied if the mesh is well refined near the interface. The purpose of Rules $2$ and $3$ is to exclude the situations illustrated in Fig.\ref{fig:3.2}, in which refinements are required to resolve the geometry of the interface. {By the Rule $4$, the three cases illustrated in Fig.\ref{fig:3.3} are not allowed since the interface elements in $\cS(K)_1$ in the left figure and in $\cS(K)_2$ in the middle and right figures are not connected, where $K$ is the dark element. We notice that the interface elements in $\cS(K)_2$ in the left figure of Fig.\ref{fig:3.3} is however connected.}

\subsection{The patterns}

{\rev{Since the interface intersects the boundary of $K$ twice at different sides (including the end points)}}, the interface intersects any element only in four possible ways as shown in Fig.\ref{fig:3.1}. We denote $\cT_1$ the set of interface elements shown in Fig.\ref{fig:3.1}(a), $\cT_2$ the set of interface elements shown in Fig.\ref{fig:3.1}(b) and (c), and $\cT_3$ the set of interface elements shown in Fig.\ref{fig:3.1}(d). By Definition \ref{def:2.1}, each element in $\cT_3$ is a large element. Thus we only need to consider the merging of type $\cT_1$ and $\cT_2$ elements.

\begin{figure}
\centering
\includegraphics[width=0.8\textwidth]{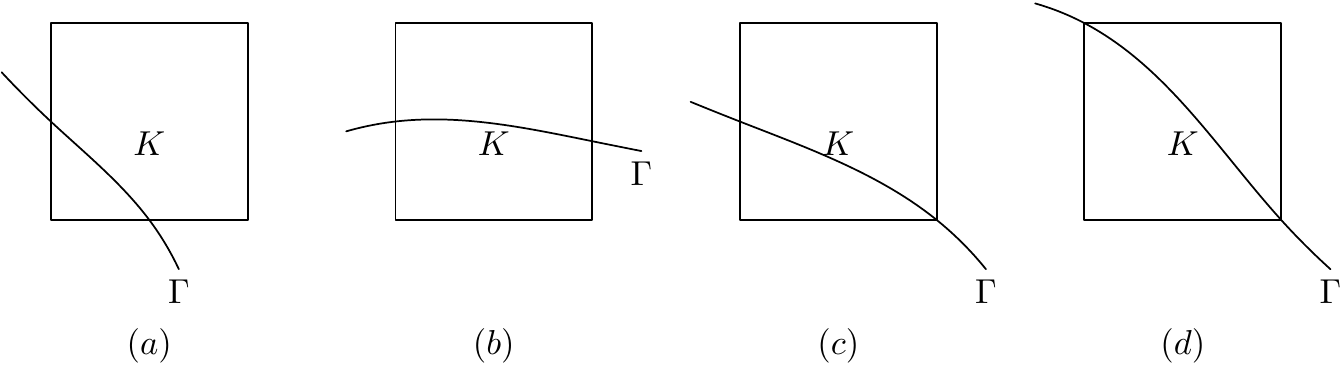}
\caption{Different types of interface elements. The type 2 elements include elements illustrated in (b) and (c). }\label{fig:3.1}
\end{figure}

A pattern is a set of interface elements and their neighboring and diagonal elements whose union consists of a macro-element. We introduce five types of patterns according to the combination of different types of interface elements, which will be used in our merging algorithm for the admissible chain of interface elements. In the following, for any $K\in \cT$, $h_i(K)$ stands for its length of the side of $K$ which is parallel to the $x_i$-axis, $i=1,2$.

\bigskip
{\bf{Pattern 1}}: $K\in\cT_1$ has two neighboring elements $N(K)_1,N(K)_2\in \cT_2$, see Fig.\ref{pattern1}. $e_1$ and $e_2$ are respectively the thick part of the sides of $N(K)_1$ and $N(K)_2$ in the figure. We use Algorithm 1 to obtain the macro-elements $M(K)$, $M(N(K)_1)$, and $M(N(K)_2)$. Here for any closed set $T\subset\R^2$, $T^\circ$ stands for the interior of $T$.

\medskip
\noindent\rule{\textwidth}{0.35mm}
\noindent{\bf{Algorithm 1:}} Pattern 1

\vspace{-0.2cm}
\noindent\rule{\textwidth}{0.35mm}

{\bf{Input:}} $(N(K)_1,K,N(K)_2)$

{\bf {Output:}} $(M(N(K)_1),M(K),M(N(K)_2))$

{\bf {if}} $K$, $N(K)_1$, and $N(K)_2$ are large elements {\bf{then}}

$\quad$ $M(N(K)_1)=N(K)_1$, $M(K)=K$, $M(N(K)_2)=N(K)_2$;

{\bf {else}}

$\quad$ $\quad$ {\bf {if}} $|e_1|/h_2(K) \ge 2\delta_0$ and $|e_2|/h_1(K) < 2\delta_0$ {\bf{then}}

$\quad$ $\quad$ $\quad$ let $M(K)=M(N(K)_1)=M(N(K)_2)=(\overline{K}\cup \overline{N(K)}_1 \cup \overline{N(K)}_2 \cup \overline{{D}(K)})^\circ$;

$\quad$ $\quad$ {\bf {else if}} {$|e_1|/h_2(K) \ge 2\delta_0$ and $|e_2|/h_1(K) < 2\delta_0$} {\bf{then}}

$\quad$ $\quad$ $\quad$ let $M(K)=M(N(K)_1)=M(N(K)_2)=(\overline{K}\cup \overline{N(K)}_1 \cup \overline{N(K)}_2 \cup \overline{{D}(K)}\cup \overline{G}_4\cup \overline{G}_5)^\circ$;

$\quad$ $\quad$ {\bf {else if}} {$|e_1|/h_2(K) < 2\delta_0$ and $|e_2|/h_1(K) \ge 2\delta_0$ } {\bf{then}}

$\quad$ $\quad$ $\quad$ let $M(K)=M(N(K)_1)=M(N(K)_2)=(\overline{K}\cup \overline{N(K)}_1 \cup \overline{N(K)}_2 \cup \overline{{D}(K)}\cup \overline{G}_1\cup \overline{G}_2)^\circ$;

$\quad$ $\quad$ {\bf {else if}} {$|e_1|/h_2(K) < 2\delta_0$ and $|e_2|/h_1(K) < 2\delta_0$} {\bf{then}}

$\quad$ $\quad$ $\quad$ let $M(K)=M(N(K)_1)=M(N(K)_2)=(\overline{K}\cup \overline{N(K)}_1 \cup \overline{N(K)}_2 \cup \overline{{D}(K)}\cup (\cup^5_{j=1}\overline{G}_j))^\circ$.

$\quad$ $\quad$ {\bf {end}}

{\bf {end}}

\vspace{-0.2cm}
\noindent\rule{\textwidth}{0.35mm}
%

\begin{figure}
  \centering
\includegraphics[width=0.6\textwidth]{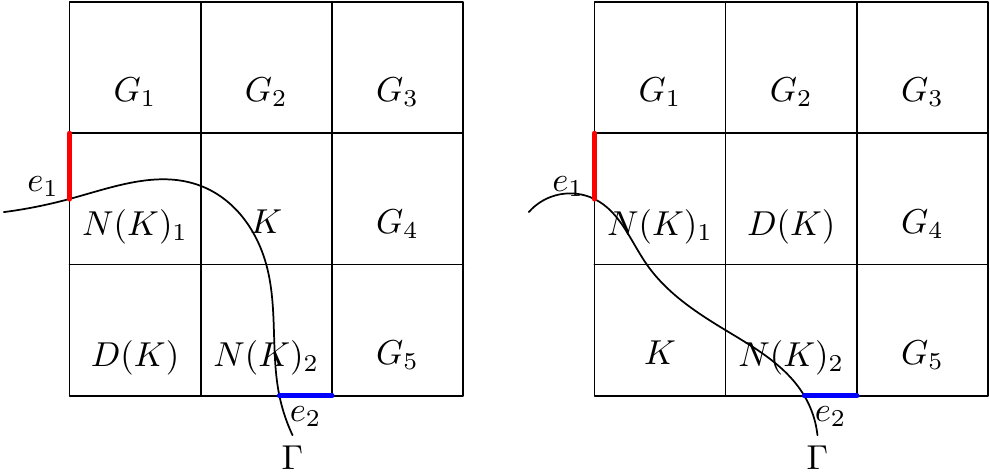}
  \caption{Illustration of type 1 (left) and type 2 (right) patterns.}\label{pattern1}
\end{figure}

\begin{lem}\label{lem:3.1}
Let $\delta_0\in(0,1/3\,]$. The macro-elements $M(K)$, $M(N(K)_1)$, $M(N(K)_2)$ of the output of Algorithm 1 are large elements.
\end{lem}
\begin{proof}
We only prove $M(K)$ is a large element when $|e_1|/h_2(K) < 2\delta_0$ and $|e_2|/h_1(K) < 2\delta_0$. The other cases can be proved analogously. Since $\delta_0\in(0,1/3\,]$, we have
\begin{align*}
\frac{|e_1|+h_2(K)}{3\,h_2(K)}\ge \frac 13\ge\delta_0,\ \ \frac{2h_2(K)-|e_1|}{3\,h_2(K)}\ge \frac 13\ge\delta_0.
\end{align*}
Similar inequalities hold for $|e_2|$. Thus $|e\cap \Omega_i|\ge \delta_0|e|$ for each side $e$ of $M(K)$ having nonempty intersection with $\Omega_i$, $i=1,2$. This implies that $M(K)$ is a large element.
\end{proof}

{\bf{Pattern 2}}: $K\in \cT_1$ has two neighboring elements $N(K)_1,N(K)_2\in \cT_1$, see Fig.\ref{pattern1}. $e_1$ and $e_2$ are respectively the thick part of the side of $N(K)_1$ and $N(K)_2$ in the figure. We use Algorithm 2 to obtain $M(K)$, $M(N(K)_1)$, and $M(N(K)_2)$.

\begin{figure}
  \centering
\includegraphics[width=\textwidth]{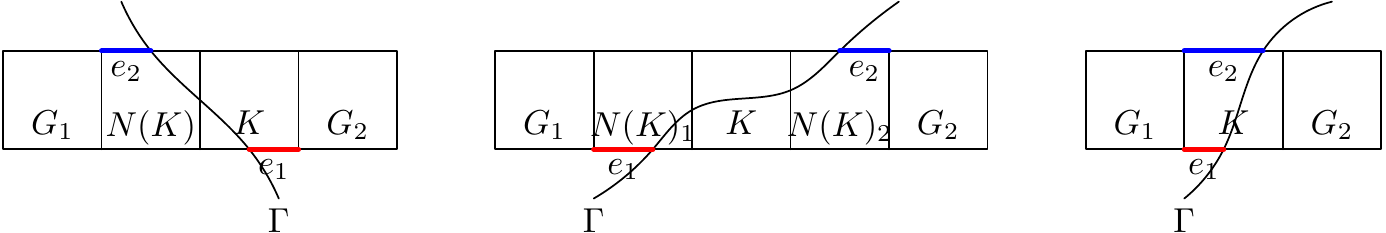}
  \caption{Illustration of type 3 (left), type 4 (middle) and type 5 (right) patterns.}\label{pattern3}
\end{figure}

\medskip
\noindent\rule{\textwidth}{0.35mm}
\noindent{\bf{Algorithm 2:}} Pattern 2

\vspace{-0.2cm}
\noindent\rule{\textwidth}{0.35mm}

{\bf{Input:}} $(N(K)_1,K,N(K)_2)$

{\bf {Output:}} $(M(N(K)_1),M(K),M(N(K)_2))$

{\bf {if}} $K,N(K)_1$, and $N(K)_2$ are large elements {\bf{then}}

$\quad$ let $M(N(K)_1)=N(K)_1$, $M(K)=K$, $M(N(K)_2)=N(K)_2$;

{\bf {else}}

$\quad$ $\quad$ {\bf {if}} $|e_1|/h_2(K) \ge 2\delta_0$ and $|e_2|/h_1(K) \ge 2\delta_0$ {\bf{then}}

$\quad$ $\quad$ $\quad$ let $M(K)=M(N(K)_1)=M(N(K)_2)=(\overline{K}\cup \overline{N(K)}_1 \cup \overline{N(K)}_2 \cup \overline{{D}(K)})^\circ$;

$\quad$ $\quad$ {\bf {else if}} {$|e_1|/h_2(K) \ge 2\delta_0$ and $|e_2|/h_1(K) < 2\delta_0$} {\bf{then}}

$\quad$ $\quad$ $\quad$ let $M(K)=M(N(K)_1)=M(N(K)_2)=(\overline{K}\cup \overline{N(K)}_1 \cup \overline{N(K)}_2 \cup \overline{{D}(K)}\cup \overline{G}_4\cup \overline{G}_5)^\circ$;

$\quad$ $\quad$ {\bf {else if}} {$|e_1|/h_2(K) < 2\delta_0$ and $|e_2|/h_1(K) \ge 2\delta_0$ } {\bf{then}}

$\quad$ $\quad$ $\quad$ let $M(K)=M(N(K)_1)=M(N(K)_2)=(\overline{K}\cup \overline{N(K)}_1 \cup \overline{N(K)}_2 \cup \overline{{D}(K)}\cup \overline{G}_1\cup \overline{G}_2)^\circ$;

$\quad$ $\quad$ {\bf {else if}} {$|e_1|/h_2(K) < 2\delta_0$ and $|e_2|/h_1(K) < 2\delta_0$} {\bf{then}}

$\quad$ $\quad$ $\quad$ let $M(K)=M(N(K)_1)=M(N(K)_2)=(\overline{K}\cup \overline{N(K)}_1 \cup \overline{N(K)}_2 \cup \overline{{D}(K)}\cup(\cup^5_{j=1} \overline{G}_j))^\circ$.

$\quad$ $\quad$ {\bf {end}}

{\bf {end}}

\vspace{-0.2cm}
\noindent\rule{\textwidth}{0.35mm}

\bigskip
{\bf{Pattern 3}}: $K\in \cT_1$ has one neighboring element $N(K)\in \cT_1$, see Fig. \ref{pattern3}. $e_1$ and $e_2$
are respectively the thick part of the side of $K$ and $N(K)$ in the figure. We use Algorithm 3 to obtain $M(K)$, $M(N(K))$.

\medskip
\noindent\rule{\textwidth}{0.35mm}
\noindent{\bf{Algorithm 3:}} Pattern 3

\vspace{-0.2cm}
\noindent\rule{\textwidth}{0.35mm}

{\bf Input:} {$(K,N(K))$}

{\bf Output:} {$(M(K),M(N(K)))$ }

{\bf if}{ $K$ and $N(K)$ are both large elements} {\bf{then}}

$\quad$ let $M(K)=K$, $M(N(K))=N(K)$\;

{\bf {else}}

$\quad$ $\quad$ {\bf {if}} $|e_1|/h_1(K) \ge 2\delta_0$ and $|e_2|/h_1(K) \ge 2\delta_0$ {\bf{then}}

$\quad$ $\quad$ $\quad$ let $M(K)=M(N(K))=(\overline{K}\cup \overline{N(K)})^\circ$;

$\quad$ $\quad$ {\bf {else if}} $|e_1|/h_1(K) \ge 3\delta_0$ {\bf{then}}

$\quad$ $\quad$ $\quad$ let $M(K)=M(N(K))=(\overline{K}\cup \overline{N(K)}\cup \overline{G}_1)^\circ$;

$\quad$ $\quad$ {\bf {else if}} $|e_2|/h_1(K) \ge 3\delta_0$ {\bf{then}}

$\quad$ $\quad$ $\quad$ let $M(K)=M(N(K))=(\overline{K}\cup \overline{N(K)}\cup \overline{G}_2)^\circ$;

$\quad$ $\quad$ {\bf {else}}

$\quad$ $\quad$ $\quad$ let $M(K)=M(N(K))=(\overline{K}\cup \overline{N(K)}\cup \overline{G}_1\cup \overline{G}_2)^\circ$.

$\quad$ $\quad$ {\bf {end}}

{\bf {end}}

\vspace{-0.2cm}
\noindent\rule{\textwidth}{0.35mm}

\bigskip
{\bf{Pattern 4}}: $K\in \cT_2$ has two neighboring elements $N(K)_1,N(K)_2\in \cT_1$, see Fig. \ref{pattern3}. $e_1$ and $e_2$ are respectively the thick part of the side of $N(K)_1$ and $N(K)_2$ in the figure. We use Algorithm 4 to obtain $M(K)$, $M(N(K)_1)$, $M(N(K)_2)$.

\medskip
\noindent\rule{\textwidth}{0.35mm}
\noindent{\bf{Algorithm 4:}} Pattern 4

\vspace{-0.2cm}
\noindent\rule{\textwidth}{0.35mm}

{\bf{Input:}}{ $(N(K)_1,K,N(K)_2)$}

{\bf{Output:}}{ $(M(N(K)_1),M(K),M(N(K)_2))$ }

{\bf {if}}{ $K$, $N(K)_1$, and $N(K)_2$ are all large elements} {\bf then}

$\quad$ {let $M(K)=K$, $M(N(K)_1)=N(K)_1$, $M(N(K)_2)=N(K)_2$;

{\bf {else}}

$\quad$ $\quad$ {\bf{if}}{ $|e_1|/h_1(K) \ge 3\delta_0$ and $|e_2|/h_1(K) \ge 3\delta_0$} {\bf{then}}

$\quad$ $\quad$ $\quad$ let $M(N(K)_1)=M(K)=M(N(K)_2)=(\overline{N(K)}_1\cup \overline{K} \cup \overline{N(K)}_2)^\circ$;

$\quad$ $\quad$ {\bf{else if}} {$|e_2|/h_1(K) \ge 4\delta_0$} {\bf{then}}

$\quad$ $\quad$ $\quad$ let $M(N(K)_1)=M(K)=M(N(K)_2)=(\overline{N(K)}_1\cup \overline{K} \cup \overline{N(K)}_2\cup\overline{G}_1)^\circ$;

$\quad$ $\quad$ {\bf{else if}} {$|e_1|/h_1(K) \ge 4\delta_0$} {\bf{then}}

$\quad$ $\quad$ $\quad$ let $M(N(K)_1)=M(K)=M(N(K)_2)=(\overline{N(K)}_1\cup \overline{K} \cup \overline{N(K)}_2\cup\overline{G}_2)^\circ$;

$\quad$ $\quad$ {\bf{else}}

$\quad$ $\quad$ $\quad$ let $M(N(K)_1)=M(K)=M(N(K)_2)=(\overline{N(K)}_1\cup \overline{K} \cup \overline{N(K)}_2\cup\overline{G}_1\cup \overline{G}_2)^\circ$.

$\quad$ $\quad$ {\bf{end}}

{\bf{end}}

\vspace{-0.2cm}
\noindent\rule{\textwidth}{0.35mm}
%

\bigskip
{\bf{Pattern 5}}: $K\in \cT_2$, see Figure \ref{pattern3}. $e_1$ and $e_2$
are respectively the thick part of the sides of $K$ in the figure. We use Algorithm 5 to obtain $M(K)$.

\medskip
\noindent\rule{\textwidth}{0.35mm}
\noindent{\bf{Algorithm 5:}} Pattern 5

\vspace{-0.2cm}
\noindent\rule{\textwidth}{0.35mm}

{\bf{Input:}}{ $K$}

{\bf{Output:}}{ $M(K)$ }

{\bf {if}}{ $K$ is a large element} {\bf then}

$\quad$ {let $M(K)=K$;

{\bf {else}}

$\quad$ $\quad$ {\bf{if}}{ $|e_1|/h_1(K) <1- 2\delta_0$ and $|e_2|/h_1(K) <1- 2\delta_0$} {\bf{then}}

$\quad$ $\quad$ $\quad$  let $M(K)=(\overline{K}\cup \overline{G}_1)^\circ$;

$\quad$ $\quad$ {\bf{else if}} {$|e_1|/h_1(K) \ge 2\delta_0$ and $|e_2|/h_1(K) \ge 2\delta_0$} {\bf{then}}

$\quad$ $\quad$ $\quad$ let $M(K)=(\overline{K}\cup \overline{G}_2)^\circ$;

$\quad$ $\quad$ {\bf{else}}

$\quad$ $\quad$ $\quad$ let $M(K)=(\overline{K}\cup \overline{G}_1\cup \overline{G}_2)^\circ$.

$\quad$ $\quad$ {\bf{end}}

{\bf{end}}

\vspace{-0.2cm}
\noindent\rule{\textwidth}{0.35mm}
%

\bigskip
The following lemma can be proved by the same argument as that in Lemma \ref{lem:3.1}. Here we omit the details.

\begin{lem}\label{lem:3.2}
The output macro-elements of Algorithm 2, Algorithm 3, Algorithm 4, and Algorithm 5 are large elements if $\delta_0\in(0,1/3\,]$, $\delta_0\in(0,1/4\,]$, $\delta_0\in(0,1/5\,]$, and $\delta_0\in(0,1/3\,]$, respectively.
\end{lem}

To conclude this subsection, we make the following observations which can be easily checked from the construction of the patterns.

\begin{rem}\label{rem0}
Only elements in $\{\cS(K)_2:K\in\cT^\Ga\}$ can be possibly merged with small interface elements. The elements two layers away from the interface will not be touched in the merging algorithm.
\end{rem}

\begin{rem}\label{rem1}
An element $G\in\cT_2$ is merged with some element $K\in\cT_1$ if and only if there exists an element $G'\in\cT_2$ such that
$G,K,G'$ form a pattern of type 1 or there exists an element $G'\in\cT_1$ such that $G,K,G'$ form a pattern of type 4.
\end{rem}

\begin{rem}\label{rem2}
An element $G\subset\Om_i$, $i=1,2$, is merged with some element $K\in\cT_1$ such that $K$ and $G$ has only one common vertex, then $G,K$, and two neighboring elements of $K$ are in the same pattern of type 1 or type 2.
\end{rem}

{{\begin{rem}\label{rem3}
If an element $G\subset \Om_i$, $i=1,2$, is merged with some element $K\in\cT^\Ga$ such that $K$ and $G$ has only one common vertex, then $K$ and $G$ are in the same pattern of type 1 or type 2. If $K$ and $G$ are in the same pattern of type 2, then $G$ can be any one of the elements $G_2,G_4,D(K)$ which has only one common vertex with some interface element in Fig.\ref{pattern1} (right). Since the interface elements in $\mathcal{S}(G)_1$ must be connected by the rule 4 of the admissible chain, $G$ cannot have any neighboring element in $\cT_2$. Thus, {\revto if $G$ has a neighboring element $N(G)\in\cT_2$ which is neighboring to $K$}, then $K, N(G),G$ are in the same pattern of type 1, which implies, in particular, that $N(G)$ is merged with $K$.
\end{rem}}}

\subsection{The merging algorithm}

Let $\mathfrak{C}$ be an admissible chain of interface elements. The following algorithm constructs a locally induced mesh from $\mathfrak{C}$ which consists of the large interface elements of $\mathfrak{C}$ and macro-elements including all small elements of $\mathfrak{C}$ so that the elements in the induced mesh are all large elements.

\noindent\rule{\textwidth}{0.35mm}
\noindent{\bf{Algorithm 6:}} The merging algorithm for the admissible chain of interface elements

\vspace{-0.2cm}
\noindent\rule{\textwidth}{0.35mm}

{\bf{Input:}} {The admissible chain $\mathfrak{C}$}

{\bf{Output:}} {The induced mesh ${\rm Induced}(\mathfrak{C})$}

 $1^\circ$ Find all subchains $\mathfrak{S}$ of length $n\ge 2$ of $\mathfrak{C}$ such that $\mathfrak{S}\{i\}\in \cT_1$, $i=1,\ldots, n$;

{\bf{if}} {$n=2k+1$ is odd} {\bf{then}}

$\quad${\bf{for}} $i=1,2,\ldots,k-1$ {\bf{do}}

$\quad$ $\quad$ call the Algorithm 3 with the input $(\mathfrak{S}\{2i\},\mathfrak{S}\{2i+1\})$;

$\quad${\bf{end}}

call the Algorithm 2 with the input $(\mathfrak{S}\{2k-1\},\mathfrak{S}\{2k\},\mathfrak{S}\{2k+1\})$

{\bf{else if}}  {$n=2k$ is even} {\bf{then}}

$\quad$ {\bf{for}} $i=1,2,\ldots,k$ {\bf{do}}

$\quad$ $\quad$ call the Algorithm 3 with the input $(\mathfrak{S}\{2i-1\},\mathfrak{S}\{2i\})$;

$\quad${\bf{end}}

{\bf{end}}

$2^\circ$ Find all subchains $\mathfrak{S}$ of length $n=3$ in the remaining interface elements such that $\mathfrak{S}\{1\}\in \cT_1$, $\mathfrak{S}\{2\}\in \cT_2$, $\mathfrak{S}\{3\}\in \cT_1$;

call the Algorithm 4 with the input $(\mathfrak{S}\{1\},\mathfrak{S}\{2\},\mathfrak{S}\{3\})$;

$3^\circ$ Find all subchains $\mathfrak{S}$ of length $n=3$ in the remaining interface elements such that $\mathfrak{S}\{1\}\in \cT_2$, $\mathfrak{S}\{2\} \in T_1$, $\mathfrak{S}\{3\}\in \cT_2$;

call the Algorithm 1 with the input $(\mathfrak{S}\{1\},\mathfrak{S}\{2\},\mathfrak{S}\{3\})$;

$4^\circ$ Find all elements $K\in \cT_2$ in the remaining interface elements;

{call the Algorithm 5 with the input $K$.}

\vspace{-0.2cm}
\noindent\rule{\textwidth}{0.35mm}

\bigskip
{\rev{ Figure \ref{fig:illustration} illustrates each step in Algorithm 6 starting from an admissible chain of interface elements. The black thin lines represent the boundaries of the elements. We remove the lines which are shared by adjacent elements in steps $1^\circ$ to $4^\circ$, meaning that two adjacent elements have been merged in the same macro-element. }}

 \begin{figure}
   \centering
   \subfigure[initial mesh]{
  \includegraphics[width=0.3\textwidth]{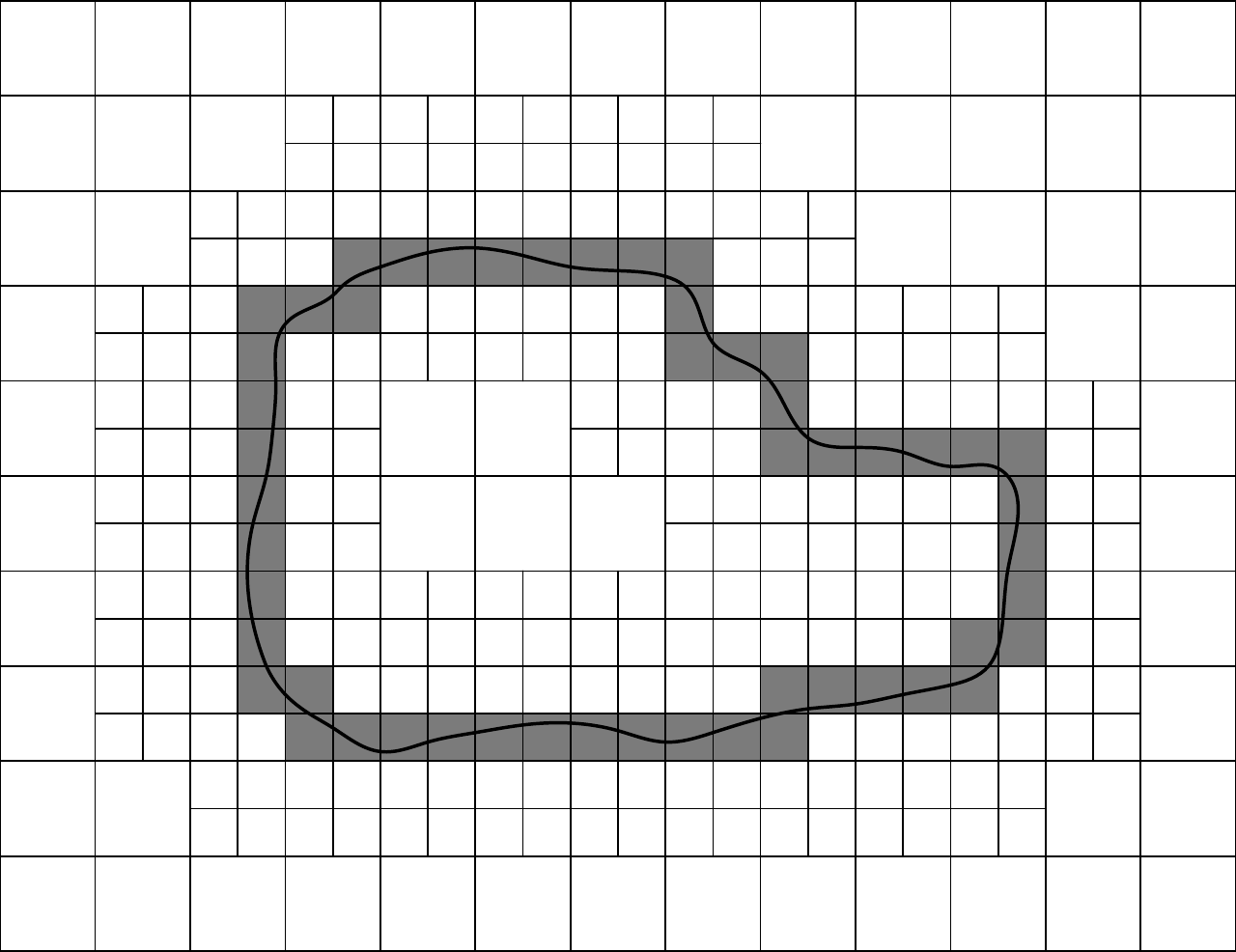}
  }
  \subfigure[step 1]{
  \includegraphics[width=0.3\textwidth]{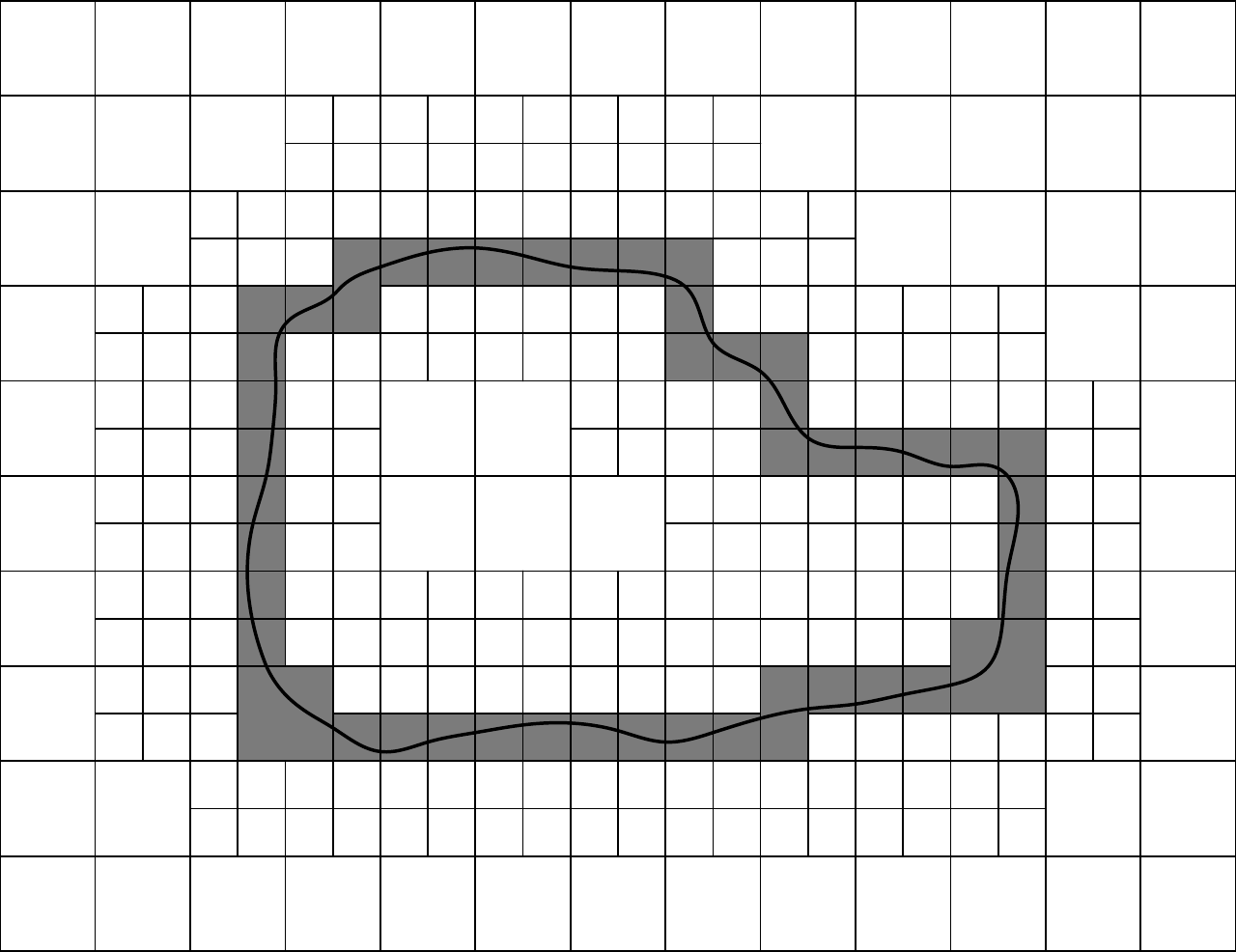}
  }
  \subfigure[step 2]{
  \includegraphics[width=0.3\textwidth]{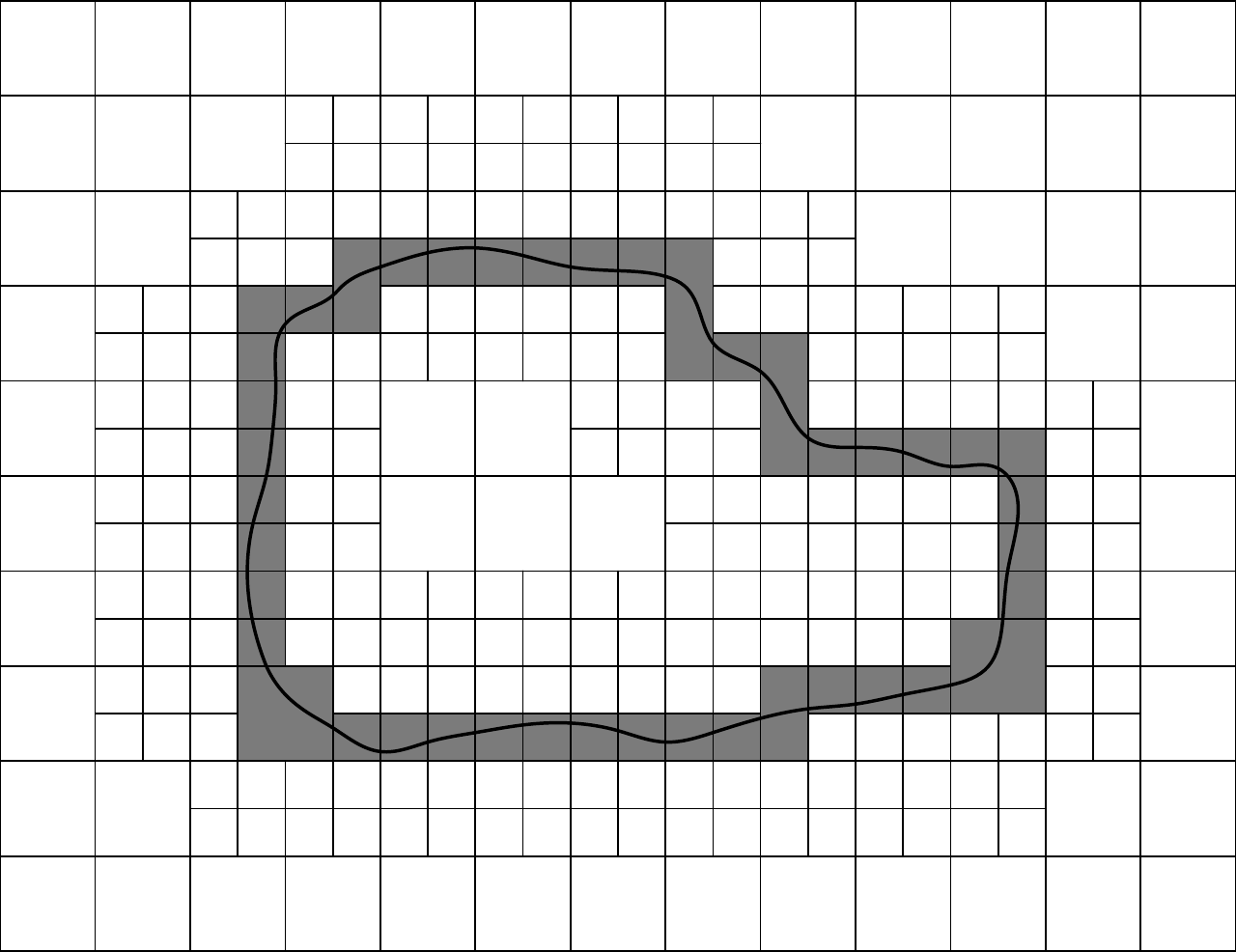}
  }
  \subfigure[step 3]{
  \includegraphics[width=0.3\textwidth]{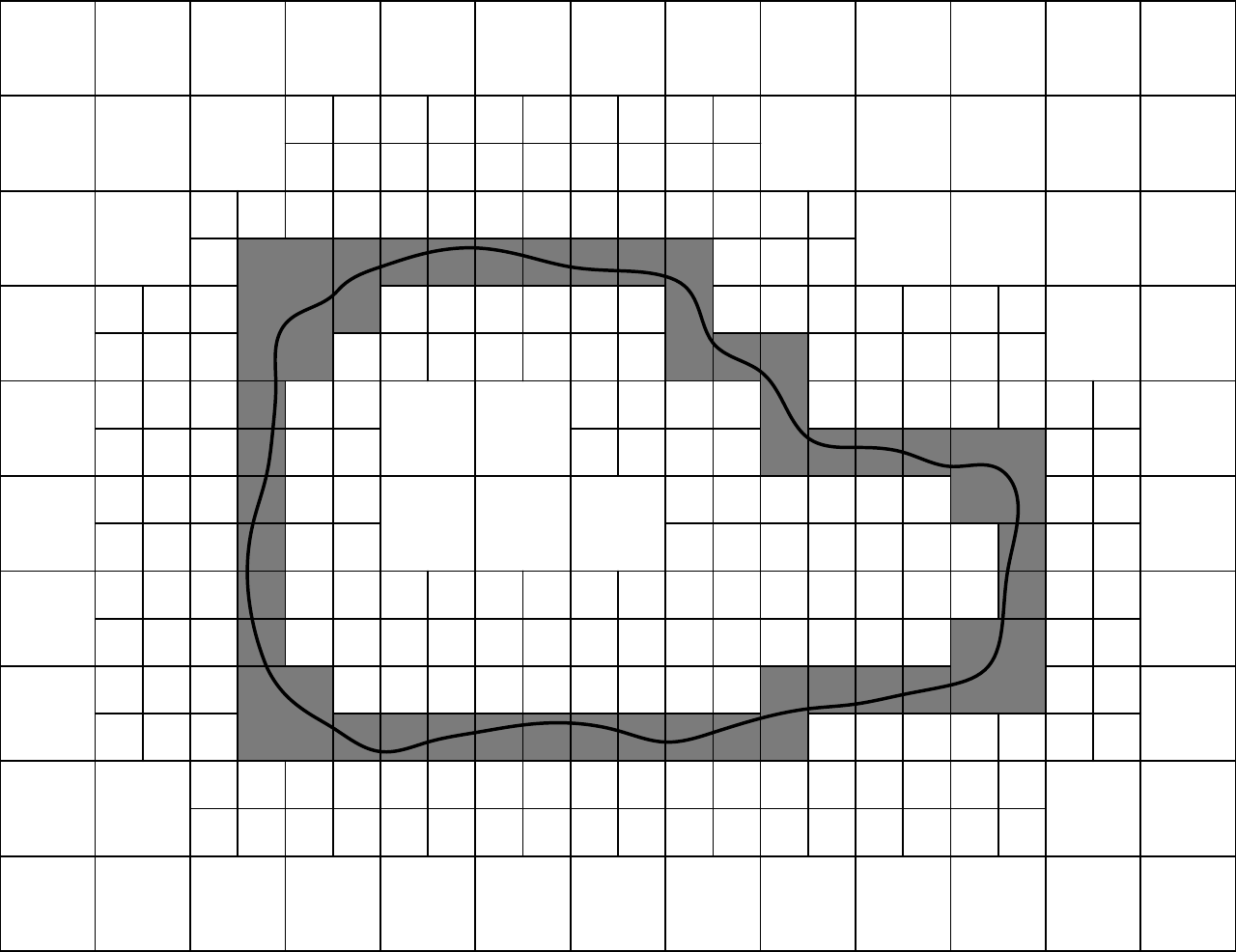}
  }
  \subfigure[step 4]{
  \includegraphics[width=0.3\textwidth]{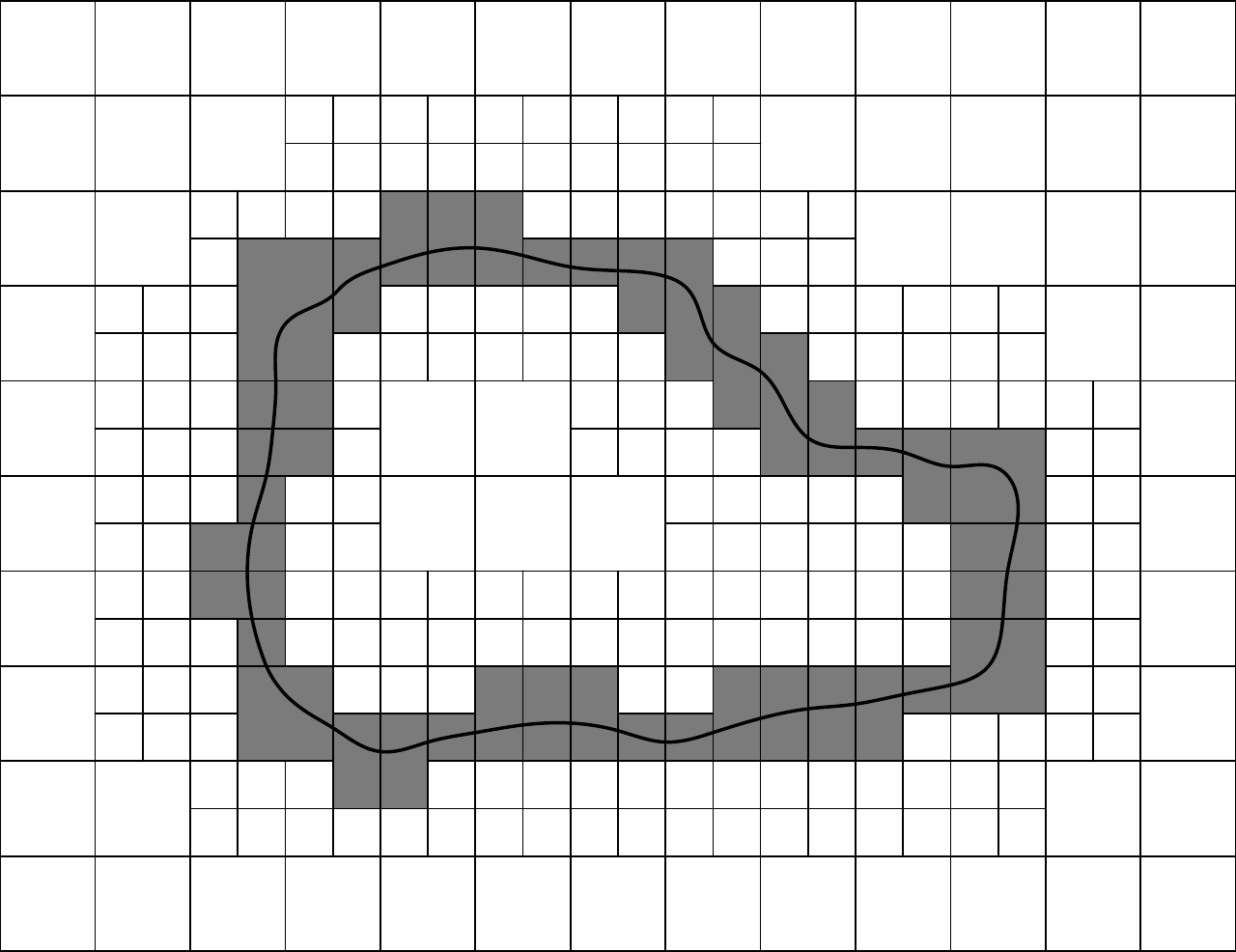}
  }\\
  \includegraphics[width=0.8\textwidth]{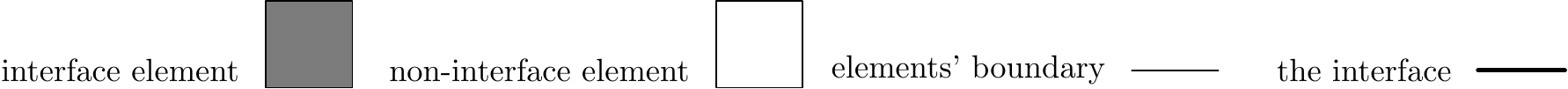}
 \caption{Illustration of the merging algorithm of the admissible chain of interface elements}\label{fig:illustration}
 \end{figure}

{ We notice that for any $K\in\cT\backslash\cT^\Ga$, the Rule 4 of the admissible chain requires that the interface elements in $\mathcal{S}(K)_1$ must be connected. These interface elements may belong to different patterns. {\revto The following lemma shows that the interface elements in $\mathcal{S}(K)_1$ belonging to the union of different patterns must be connected if $K$ belongs to these patterns.} The proof indicates that the order of merging different types of patterns in Algorithm 6 is crucial. The lemma will be used in our proof of the reliability of Algorithm 6.

\begin{lem}\label{lem:3.3} Let $\mathfrak{C}$ be an admissible chain of interface elements of length $n\ge 2$. If $K\in\cT\backslash\cT^\Ga$ is merged with interface elements in $\mathcal{S}(K)_1$ which belong to two different patterns $\cP_1$ and $\cP_2$ by Algorithm 6, then the interface elements in $(\cP_1\cup\cP_2)\cap\cS(K)_1$ are connected.
\end{lem}

\begin{proof} Denote $\cP_j^\Ga:=\cP_j\cap(\mathcal{S}(K)_1\cap\cT^\Ga)$, $j=1,2$, the interface elements of $\cP_j$ in $\mathcal{S}(K)_1$. Let $\dist(\cP_1^\Ga,\cP_2^\Ga)=\min_{D_1\in\cP_1^\Ga,D_2\in\cP_2^\Ga}\dist(D_1,D_2)$, where $\dist(D_1,D_2)$ is the minimum number of non-interface elements connecting $D_1,D_2$ in $\mathcal{S}(K)_1\backslash\mathcal{S}(K)_0$. Clearly, $0\le\dist(\cP_1^\Ga,\cP_2^\Ga)\le 3$ and $\dist(\cP^\Ga_1,\cP^\Ga_2)=0$
implies the interface elements in $(\cP_1\cup\cP_2)\cap\cS(K)_1$ are connected. We now show that $\dist(\cP_1^\Ga,\cP_2^\Ga)\not=0$ is impossible in three steps.

$1^\circ$ ${\rm dist} (\cP_1^\Gamma,\cP_2^\Gamma)=1$. Let $D_1\in\cP_1^\Ga$, $D_2\in\cP_2^\Ga$, ${\rm dist}(D_1,D_2)=1$. By $S(K)_1\cap\cT^\Ga$ is connected and the Rule 3 of the admissible chain, we know that $D_3$, which is the neighboring element to $D_1$ and $D_2$, is in $\cT^\Ga$ and $D_3$ can be either neighboring to $K$ or diagonal to $K$.

(1) If $D_3$ is diagonal to $K$, by the Rule 2 of the admissible chain, $D_3\in \cT_1$, see Fig. \ref{fig1_lemma}. Firstly assume $D_1,D_2\in \cT_1$, see Fig. \ref{fig1_lemma} (a), then there are three elements $D_1\to D_3\to D_2$ forming a chain and all elements belong to $\cT_1$. By our merging Algorithm, $D_1,D_3$ or $D_2,D_3$ will form a pattern of type 3 which will be merged by Algorithm 3 before they can be merged with other interface elements forming a pattern of type 2. But these type 3 patterns will not use $K$, which contradicts to the assumption that $K$ is merged with $D_1$ and $D_2$.

Secondly assume $D_1,D_2\in \cT_2$, see Fig. \ref{fig1_lemma} (b). By the Rule 3 of the admissible chain, $D_5,D_6\in \cT\backslash\cT^\Ga$. If $K$ is merged with its neighboring element $D_1\in \cT_2$ then $D_1$ and $K$ will be in a pattern of type 1 or 5. When $\cP_1$ is a pattern of type 1, since $D_3$ is not merged with $K$, then $D_7\in \cT_1$, $D_8\in \cT_2$, $D_1,D_7,D_8$ can form a pattern of type 1 and merged with $K$. However, in this case,  $D_7,D_1,D_3$ will form a pattern of type 4 which will be merged by Algorithm 4 before $D_1,D_7,D_8$ are merged by Algorithm 1 in the second step of our merging Algorithm. Thus $\cP_1$ cannot be of type 1. Similarly, $\cP_2$ also cannot be of type 1. The remaining case is that $\cP_1$ and $\cP_2$ are both patterns of type 5. But this case is also impossible because $D_1,D_3,D_2$ will form a pattern of type 1 which will be merged by Algorithm 1 before $D_1$, $D_2$ can possibly be merged with $K$ by Algorithm 5 in the third step of our merging Algorithm.

Finally assume $D_1\in \cT_2$, $D_2\in \cT_1$, see Fig. \ref{fig1_lemma} (c). In this case $D_3$ and $D_2$ will be in a pattern of type 2 or 3 which will be merged by Algorithm 2 or 3 in the first step of our merging Algorithm. In both cases, they will not use the element $K$, which contradicts to the assumption that $K$ is merged with $D_2$.

\begin{figure}
  \centering
   \subfigure[]{
  \includegraphics[width=0.25\textwidth]{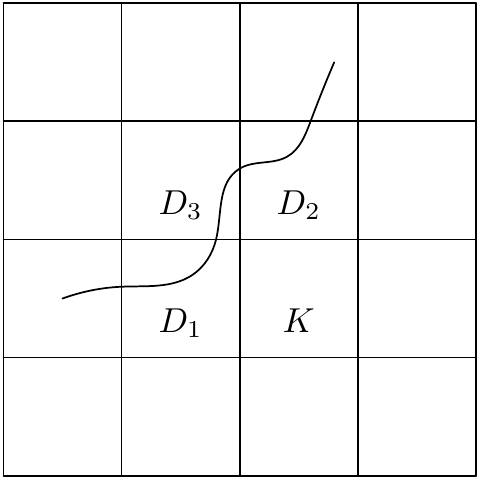}
  }\label{fig1_lemma_1}
  \subfigure[]{
  \includegraphics[width=0.25\textwidth]{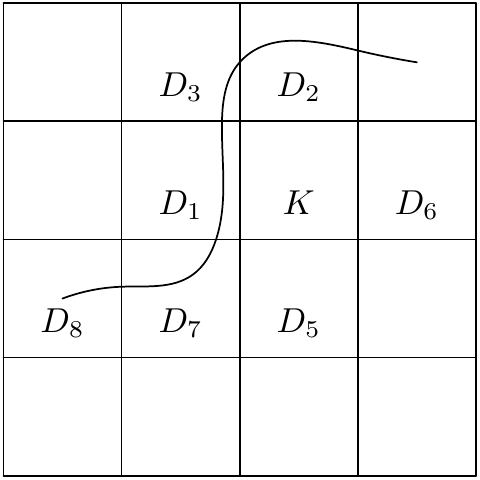}
  }\label{fig1_lemma_2}
  \subfigure[]{
  \includegraphics[width=0.25\textwidth]{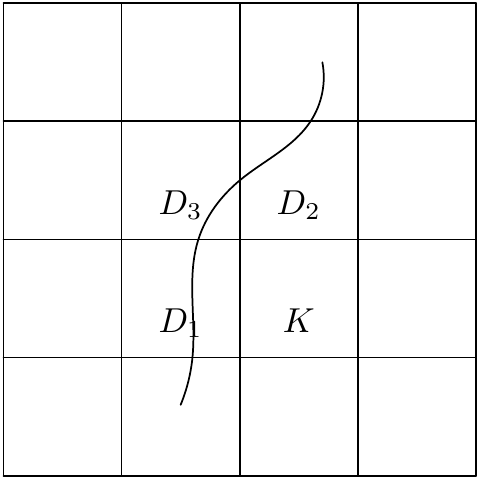}
  }\label{fig1_lemma_3}
  \caption{The element $K$ and $D_1,D_2\in N(K)$.}\label{fig1_lemma}
\end{figure}

(2) If $D_3$ is neighboring to $D_1,D_2$, again since $\mathcal{S}(K)_1\cap\cT^\Ga$ is connected and by the Rule 3 of the admissible chain, $D_3\in\cT_2$, see Fig.\ref{fig2_lemma} (a).
Since $K$ is merged with its diagonal elements $D_1$, and $K$ has a neighboring element $D_3\in \cT_2$, by Remark \ref{rem3}, $D_3$ is merged with $K$. This contradicts to ${\rm dist}(\cP_1,\cP_2)=1$.

$2^\circ$ ${\rm dist} (\cP_1^\Gamma,\cP_2^\Gamma)=2$, see Fig. \ref{fig2_lemma} (b). Since $S(K)_1\cap \cT^\Ga$ is connected, we have $D_3\in \cT_2$, $D_4\in \cT_1$. Again by Remark \ref{rem3}, $D_3$ is merged with $K$, which contradicts to ${\rm dist}(\cP_1,\cP_2)=2$.

\begin{figure}
  \centering
   \subfigure[]{
  \includegraphics[width=0.25\textwidth]{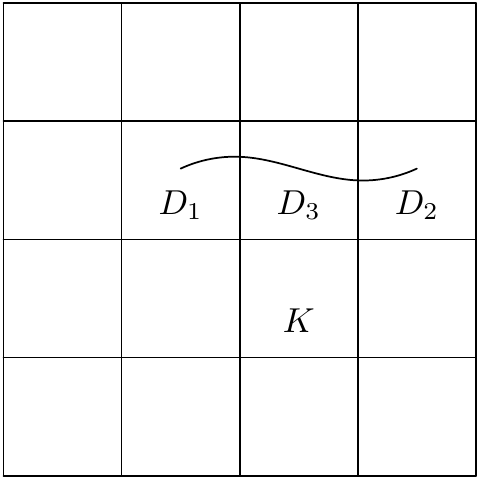}
  }
  \subfigure[]{
  \includegraphics[width=0.25\textwidth]{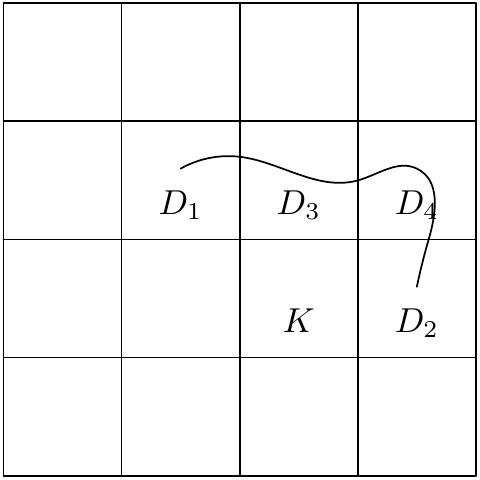}
  }
  \caption{The element $K$ and $D_1,D_2\in D(K)$ (left); The element $K$ and ${\rm dist}(D_1,D_2)=2$ (right).}\label{fig2_lemma}
\end{figure}

$4^\circ$ ${\rm dist} (\cP_1^\Gamma,\cP_2^\Gamma)=3$, see Fig. \ref{fig3_lemma}. There are two possibilities:

(1) $D_1,D_2$ are diagonal to $K$, see Fig. \ref{fig3_lemma} (a). Then $D_3\in \cT_2$, $D_4\in \cT_1$, and $D_5\in \cT_2$. By Remark \ref{rem3}, $D_3,D_5$ are merged with $K$, which contradicts to ${\rm dist}(\cP_1^\Ga,\cP_2^\Ga)=3$.

(2) $D_1,D_2$ are neighboring to $K$, see Fig. \ref{fig3_lemma} (b). This case is also impossible because $S(K)_1\cap \cT^\Ga$ is connected, then it will lead to $K$ has 3 neighboring elements in $\cT^\Ga$ which contradicts to the Rule 3 of the admissible chain.
This completes the proof.
\end{proof}}

\begin{figure}
  \centering
   \subfigure[]{
  \includegraphics[width=0.25\textwidth]{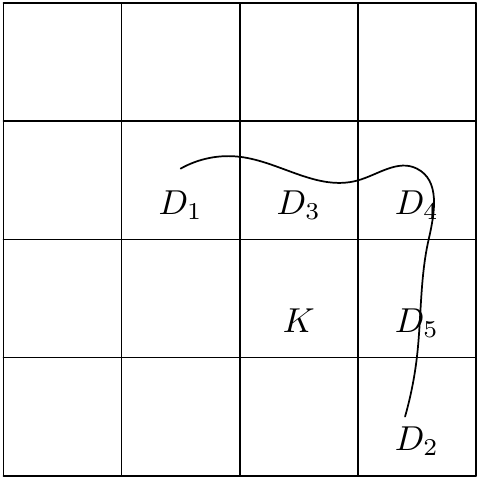}
  }
  \subfigure[]{
  \includegraphics[width=0.25\textwidth]{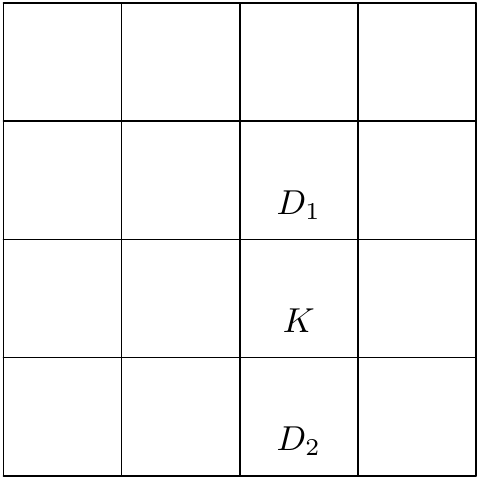}
  }
  \caption{The element $K$ and $D_1,D_2\in D(K),\, {\rm dist}(D_1,D_2)=3$ (left); The element $K$ and $D_1,D_2\in N(K),\, {\rm dist}(D_1,D_2)=3$ (right).}\label{fig3_lemma}
\end{figure}

We attach any chain of interface elements $\mathfrak{C}$ of length $n\ge 1$ an accompany chain $\mathfrak{N}(\mathfrak{C})=\{N_1\to N_2\to\cdots\to N_n\}$ with $N_i=1$ or $2$ according to $\mathfrak{C}\{i\}\in\cT_1$ or $\cT_2$, $i=1,\cdots,n$. The following theorem shows the reliability of the merging algorithm.

\begin{thm}\label{thm:3.1}
Let $\delta_0\in(0,1/5\,]$. For any admissible chain of interface elements $\mathfrak{C}$ {\rev {with length $n\geq2$, if $\mathfrak{C}(1), \mathfrak{C}(n)\in\cT_2$ or $\mathfrak{C}(1)=\mathfrak{C}(n)$, then Algorithm 6 terminates in finite number of steps with input $\mathfrak{C}$.}} All elements of the locally induced mesh ${\rm Induced}(\mathfrak{C})$ are large elements.
\end{thm}

\begin{proof} By the step $1^\circ$ of the algorithm, any two consecutive elements of type $\cT_1$ are merged. Thus in the remaining elements of the chain, the type $\cT_1$ elements must be interlaced if they are present. The step $2^\circ$ merges all remaining elements in the chain which consists of a subchain of length 3 of the type $1\to 2\to 1$. The remaining type $\cT_1$ elements in the chain of length $3$ can appear only in the form $2\to 1\to 2$ which are merged by the step $3^\circ$. {\rev{Thus the first three steps of the algorithm merge all elements in $\cT_1$. Here we have used the assumption that the first and last elements in $\mathfrak{C}$ both belong to $\cT_2$ or the first and last elements are the same interface elements.}} The left type $\cT_2$ elements are treated in the step $4^\circ$ of the algorithm. The elements in $\cT_3$ are all large elements and thus need not be merged.
This shows that Algorithm 6 will merge all interface elements in the chain to output a locally induced mesh ${\rm Induced}(\mathfrak{C})$ which consists of the large elements of $\mathfrak{C}$ and the macro-elements containing all small elements of the chain $\mathfrak{C}$. By Lemmas \ref{lem:3.1}-\ref{lem:3.2}, the elements in ${\rm Induced}(\mathfrak{C})$ are all large elements since $\delta_0\in (0,1/5\,]$.

It remains to show that the non-interface elements of the mesh $\cT$ will not be used twice in the merging Algorithm 6 to guarantee the success of the algorithm. Let $K\in\cT\backslash\cT^\Gamma$. We {first} assume $K$ is merged with interface elements in $\cS(K)_1$ which belong to two patterns $\cP_1$ and $\cP_2$. {By Lemma \ref{lem:3.3}, the elements in $(\cP_1\cup\cP_2)\cap\cS(K)_1$ must be connected}. Assume $D_1\in\cP_1$, $D_2\in\cP_2$ are connected, then one of $D_1$ and $D_2$ must be diagonal to $K$. Without loss of generality, we assume $D_1=D(K)$. If $D_1\in\cT_2$, then by Remark \ref{rem1} the element $D'$ neighboring $K$ and $D_1$ must be in $\cT_1$ so that $K,D_1',D_1$ form a pattern of type 1. Thus by Rule 2 of the admissible chain, $D_2$, as an interface element, cannot be neighboring $D_1$, see Fig.\ref{fig:SK1} (top left). This is a contradiction. Therefore, $D_1$ can only be of type $\cT_1$. There are three possibilities illustrated in Fig.\ref{fig:SK1}.

(1) In the case of Fig.\ref{fig:SK1} (top right), Rule 2 implies $D_2$ must be in $\cT_2$. By Remark \ref{rem1}, $K,D_1$, and $D_2$ form a pattern of type 1, which contradicts to the assumption that $D_1,D_2$ belong to different patterns.

(2) In the case of Fig.\ref{fig:SK1} (bottom left), Rule 2 implies $D_2$ cannot be neighboring $D_1$.

(3) {\rev{ In the case of Fig.\ref{fig:SK1} (bottom middle), $D_1$ has only one common vertex with $K$. By Remark \ref{rem2} the neighboring elements of $D_1,D_1',D_1''$ both must be of type $\cT_1$ or $\cT_2$ and $K,D_1,D_1',D_1''$ form a pattern of type 1 or 2. If $K,D_1,D_1',D_1''$ form a pattern of type $1$, then this case belongs to (1). If $K,D_1,D_1',D_1''$ form a pattern of type 2, $D_2$ must be equal to one of $D_1'$ or $D_1''$, which contradicts that $D_1,D_2$ are in different patterns.

\begin{figure}[h]
\centering
%
\includegraphics[width=0.8\textwidth]{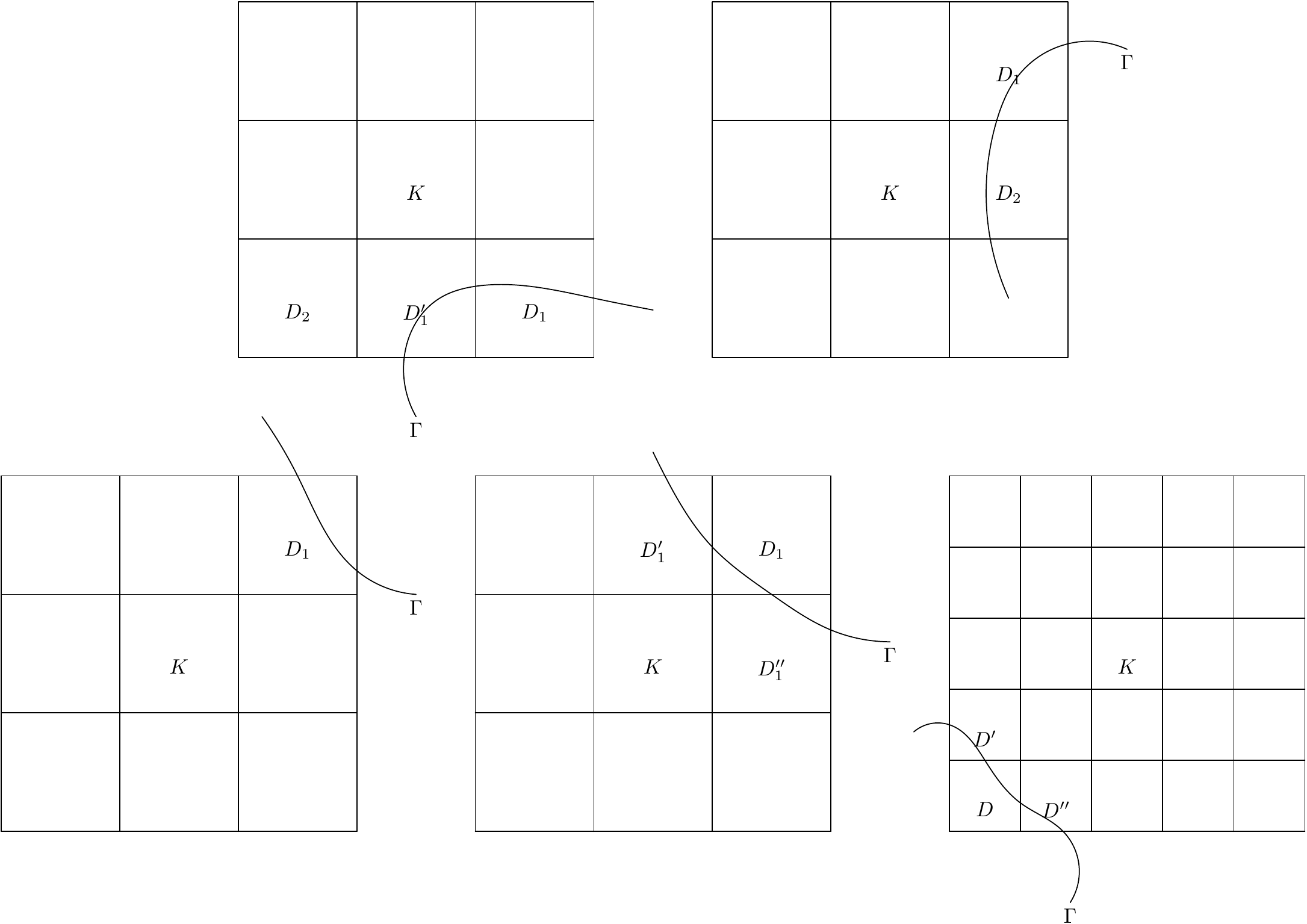}
 \caption{The element $K$ and $D_1\in\cS(K)_1$ is a type $\cT_2$ element (top left). The element $K$ and $D_1\in\cS(K)_1$ is a type $\cT_1$ element (top right, bottom left, and bottom middle). The element $K$ and $D,D',D''$ in $\cS(K)_2\backslash\cS(K)_1$ (bottom right).}\label{fig:SK1}
\end{figure}

In conclusion, $K$ cannot be merged with two interface elements in $\cS(K)_1$ belonging to different patterns.}} {It remains to show that $K$ cannot be merged with different interface elements in $\mathcal{S}(K)_2$ belonging to two patterns $\cP_1,\cP_2$ of which one pattern, e.g., $\cP_1$, consists of only interface elements in $\mathcal{S}(K)_2\backslash\mathcal{S}(K)_1$. By the construction of the patterns in \S 3.2, $\cP_1$ must be a pattern of type 2 and all interface elements $D,D',D''$ in $\cP_1$ are in the second layer of elements surrounding $K$, see Fig.\ref{fig:SK1} (bottom right). In this case, we know by the Rule 4 of the admissible chain that $K$ cannot be merged with elements in $\cS(K)_2$ other than $D,D',D''$, that is, $K$ cannot be merged with interface elements belonging to the second pattern $\cP_2$. This completes the proof.}
\end{proof}


To conclude this section, we show that the merging Algorithm 6 leads to a reliable algorithm to automatically construct a body-fitted shape regular mesh for arbitrarily shaped smooth interface. We start from a conforming uniform mesh $\cT_0$ of the domain $\Om$. We refine the interface elements of $\cT_0$ by quad refinements and their surrounding elements to generate a Cartesian mesh $\cT$ with hanging nodes such that all interface elements of $\cT$ form an admissible chain $\mathfrak{C}$. This is possible because the interface $\Ga$ is $C^2$-smooth. Now we use Algorithm 6 to obtain an induced mesh $\cM={\rm Induced}(\mathfrak{C})$. Since each interface element $K\in\cM^\Ga$ is a large element, $K_i^h$, $i=1,2$, is strongly shape regular in the sense that it is the union of shape regular triangles which we denote as $T_K^{ij}$, $1\le j\le m_K$. Then the mesh
\ben
\widetilde{\cM}=\{T_K^{ij}: i=1,2,\ j=1\cdots, m_K,K\in\cM^\Ga\}\cup\{K:K\in\cM\backslash\cM^\Ga\}
\een
is a triangular-rectangular mixed finite element mesh of the domain $\Om$. $\{T_K^{ij}: i=1,2,\ j=1\cdots, m_K,K\in\cM^\Ga\}$ is a body-fitted shape regular triangular mesh that covers the interface and $\{K:K\in\cM\backslash\cM^\Ga\}$ consists of a rectangular mesh whose elements are similar to the elements of the initial mesh $\cT_0$. Fig.\ref{fig:new} shows a mixed mesh constructed from the unfitted finite element mesh in Fig.\ref{fig:illustration}(e).

\begin{figure}
\begin{center}
 \includegraphics[width=0.4\textwidth]{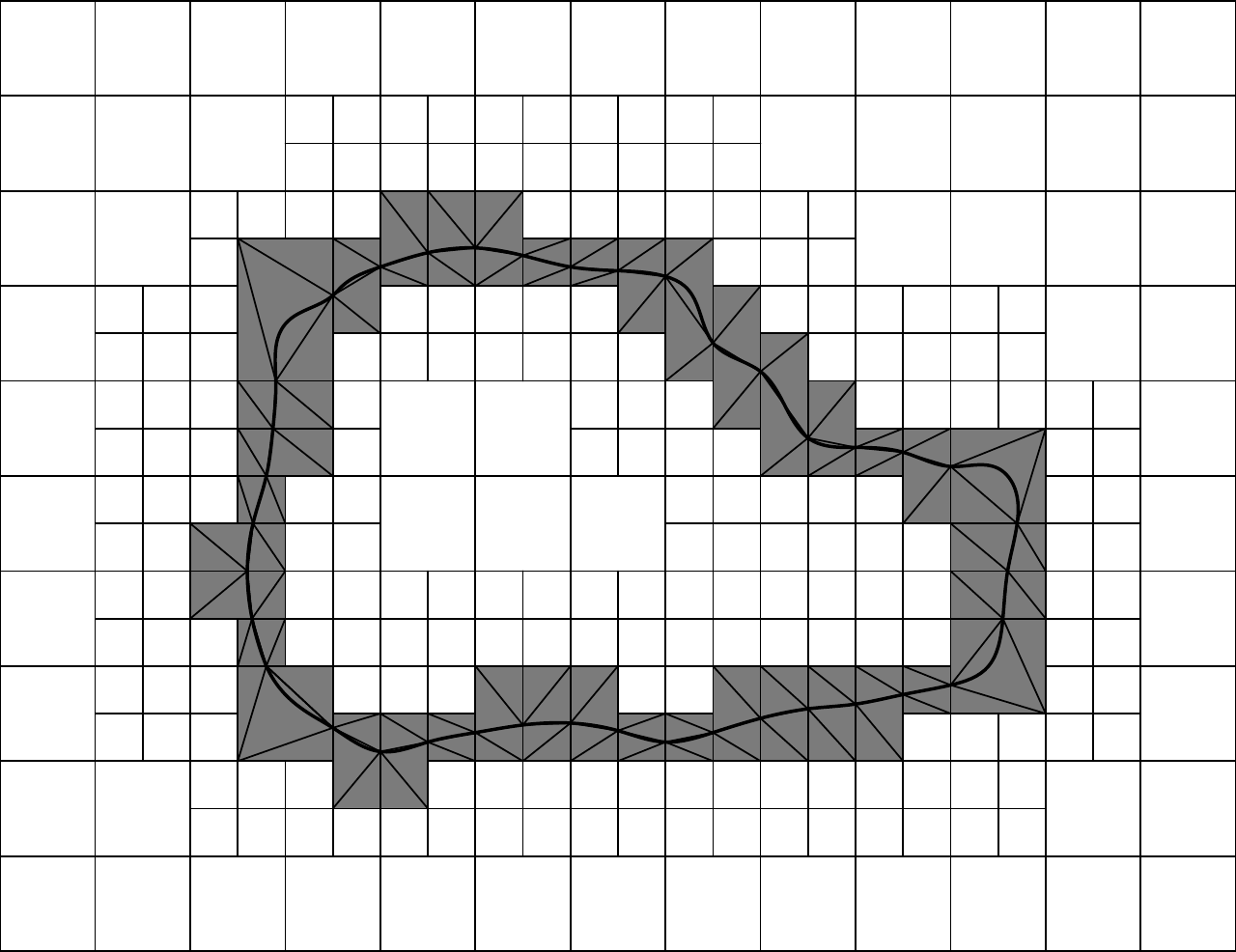}
 \end{center}
\caption{Illustration of a mixed triangular-rectangular body-fitted shape regular finite element mesh.}\label{fig:new}
\end{figure}

\section{The condition number of the stiffness matrix}

In this section we study the condition number of the stiffness matrix of the unfitted finite element method
defined in \eqref{a2}. Since we allow the Cartesian mesh $\cT$ having hanging nodes, which is a nonconforming mesh in the classical sense, we recall an important concept of the $K$-mesh in Babu\v{s}ka and Miller \cite{Babuska87a}. It is
introduced to control the undesirable excessive local refinements so that the local mesh sizes around each vertex of the elements are comparable. This concept is further developed in Bonito and Nochetto \cite{Bonito} as the control of the local level of incompatibility of the nonconforming meshes.

Let $\mathcal{N}^0$ be the set of conforming nodes of the mesh $\cT$. A conforming node of $\cT$ is a vertex of the elements in $\cT$ which either locates on the boundary $\pa\Om$ or is shared by the four elements to which it belongs. For each conforming node $P$, we define $\psi_P\in\mathbb{X}_1(\cT)\cap H^1(\Om)$, which is bilinear in each element and satisfies $\psi_P(Q)=\delta_{PQ}$ for any $Q\in\mathcal{N}^0$. Here $\delta_{PQ}$ is the Kronecker delta. It is proved in \cite{Babuska87a} that $\{\psi_P:P\in\mathcal{N}^0\}$ consists of a basis of $\X_1(\cT)\cap H^1(\Om)$ and satisfies the property of the partition of unity
$\sum_{P\in\mathcal{N}^0}\psi_P=1$.
In the rest of the paper, we impose the following assumption on the finite element mesh $\cT$ which is called the $K$-mesh in \cite{Babuska87a}.

\medskip
{\bf Assumption (H3)} There exists a constant $C>0$ uniform on the level of discretization of $\cT$ such that for any conforming node $P\in\mathcal{N}^0$,
\ben
{\rm diam}({\rm supp}(\psi_P))\le C\min_{K\in\cT_P}h_K,
\een
where $\cT_P:=\{K\in\cT,\,\,K\subset{\rm supp}(\psi_P)\}$.
\medskip

One can find further properties of $K$-meshes in \cite{Babuska87a}. We refer to \cite[\S 6]{Bonito} for a refinement algorithm to enforce the assumption (H3) in practical computations.

The following lemma on the continuous approximation of discontinuous piecewise polynomials on $K$-meshes is proved in \cite[Lemma 3.2]{CLX}.

\begin{lem}\label{lem:2.2}
Let $\mathbb{V}_P(\cT)=\Pi_{K\in\cT} Q_p(K)$. There exists an interpolation operator $\pi_h:\V_p(\cT)\to\mathbb{V}_p(\cT)\cap H^1(\Om)$ such that for any $v\in\V_p(\cT)$,
\ben
& &\|v-\pi_h v\|_{L^2(K)}\le C\|p^{-1}h^{1/2}\lj v\rj\|_{L^2(\sigma(K))},\\
&&\|\na(v-\pi_h v)\|_{L^2(K)}\le C\|ph^{-1/2}\lj v\rj\|_{L^2(\sigma (K))},
\een
where $\sigma(K)=\{e\in\cE^{\rm side}:e\subset\widetilde\omega(K)\}$, $\widetilde\omega(K)$ is a set of elements including $K$ such that ${\rm diam}(\widetilde\omega (K))\le Ch_K$. The constant $C$ is independent of $h_K,p$. Moreover, $\pi_hv\in H^1_0(\Om)$ if $v=0$ on $\pa\Om$.
\end{lem}

Since the induced mesh $\cM={\rm Induced}(\cT)$ is obtained by merging some of the elements of $\cT$, $\V_p(\cM)\subset\V_p(\cT)$. Thus Lemma \ref{lem:2.2} is also valid for any functions $v\in\V_p(\cM)$. We have the following discrete Poincar\'e inequality.

\begin{lem}\label{lem:2.3}
For any $v\in\X_p(\cM)$, we have $\|v\|_{L^2(\Om)}\le C\|v\|_{\rm DG}$, where $C>0$ is a constant independent of the mesh sizes, $p$, and the interface deviations $\eta_K$ for all $K\in\cM^\Ga$.
\end{lem}

\begin{proof} Let $v=v_1\chi_{\Om_1}+v_2\chi_{\Om_2}\in\X_p(\cM)$. By Lemma \ref{lem:2.2}, for $v_i\in\V_p(\cM_i)$, $i=1,2$, there exists $\pi_h v_i\in\V_P(\cM_i)\cap H^1(\Om_i^h)$ such that
\beq
\|v_i-\pi_h v_i\|_{\cM_i}\le C\|p^{-1}h^{1/2}\lj v_i\rj\|_{{\rev{\cE_i^{side}}}},\ \ \|\na(v_i-\pi_h v_i)\|_{\cM_i}\le C\|ph^{-1/2}\lj v_i\rj\|_{{\rev{\cE_i^{side}}}}. \label{a4}
\eeq
Recall that we have assumed $\bar\Om_1\subset\Om$. Let $w_i\in H^1(\Om_i)$, $i=1,2$, satisfy
\ben
& &-\Delta w_1=0\ \ \mbox{in }\Om_1, \ \ w_1=\lj\pi_h v\rj_\Ga\ \ \mbox{on }\Ga=\pa\Om_1,\\
& &-\Delta w_2=0\ \ \mbox{in }\Om_2, \ \ w_2=0\ \ \mbox{on }\Ga,\ \ w_2=\pi_h v_2\ \ \mbox{on }\pa\Om.
\een
Then $w_i\in H^1(\Om_i)$ satisfies $\|w_1\|_{H^1(\Om_1)}\le C\|\lj\pi_h v\rj\|_{H^{1/2}(\Ga)}$, $\|w_2\|_{H^1(\Om_2)}\le C\|\pi_h v\|_{H^{1/2}(\pa\Om)}$. From the proof of \cite[Lemma 3.4]{CLX} we know that
\beq
\|\lj\pi_h v\rj\|_{H^{1/2}(\Ga)}^2\le C(\|ph^{-1/2}\lj v\rj\|_{\cE^\Ga\cup\cE^{\rm side}_1\cup\cE^{\rm side}_2}+\|p^{-1}h^{1/2}\na_T\lj v\rj\|_{\cE^\Ga}).\label{a5}
\eeq
Now we use a similare argument to bound $\|\pi_h v\|_{H^{1/2}(\pa\Om)}$. By the the localization lemma of the $H^{1/2}$ semi-norm in Faermann \cite[Lemm 2.3]{Faermann} and the Gagliardo-Nirenberg type estimate for the $H^{1/2}$ semi-norm, we obtain as in \cite[(3.13)]{CLX} that
\be
&&\|\pi_h v_2\|_{H^{1/2}(\pa\Om)}^2\le C\sum_{K\in\cM_2}(\|\pi_h v_2\|_{L^2(\Sigma_K)}\|\na_T(\pi_h v_2)\|_{L^2(\Sigma_K)}+h^{-1}_K\|\pi_h v_2\|_{L^2(\Sigma_K)}^2),\label{a6}
\ee
where $\Sigma_K=\pa K\cap\pa\Om$. By the $hp$-inverse estimate and Lemma \ref{lem:2.2}, we obtain
\ben
\|\na_T(\pi_h v_2)\|_{L^2(\Sigma_K)}&\le&\|\na_T v_2\|_{L^2(\Sigma_K)}+\|{\rev{\na_T}}(v_2-\pi_h v_2)\|_{L^2(\Sigma_K)}\\
&\le&Cp^2h_K^{-1}\|v_2\|_{L^2(\Sigma_K)}+Cph_K^{-1/2}\|\na(v_2-\pi_h v_2)\|_{L^2(K)}\\
&\le&Cp^2h_K^{-1}\|v_2\|_{L^2(\Sigma_K)}+Cph_K^{-1/2}\|ph^{-1/2}\lj v_2\rj\|_{L^2(\sigma(K))}.
\een
Similarly, one can prove $\|\pi_h v_2\|_{L^2(\Sigma_K)}\le\|v_2\|_{L^2(\Sigma_K)}+\|\lj v_2\rj\|_{L^2(\sigma(K))}$. Recall that $\lj v_2\rj=v_2$ on $\pa\Om$. This implies by \eqref{a6} that
\ben
\|\pi_h v_2\|_{H^{1/2}(\pa\Om)}\le C\|ph^{-1/2}\lj v_2\rj\|_{\cE^{\rm bdy}\cup\cE_2^{\rm side}}.
\een
Therefore, by combining with \eqref{a5} we have
\beq
\|w_1\|_{H^1(\Om_1)}+\|w_2\|_{H^1(\Om_2)}\le C\|v\|_{\rm DG}.\label{a7}
\eeq
Let $\pi_h^cv=(\pi_h v_1-w_1)\chi_{\Om_1}+(\pi_h v_2-w_2)\chi_{\Om_2}$. Then $\pi_h^cv\in H^1_0(\Om)$ and by using Poincar\'e inequality for $\pi_h^cv$, we have
\ben
\|v\|_{L^2(\Om)}&\le&\|v-\pi_h^c v\|_{L^2(\Om)}+\|\pi_h^c v\|_{L^2(\Om)}\\
&\le&\sum^2_{i=1}(\|v_i-\pi_h v_i\|_{\cM_i}+\|w_i\|_{L^2(\Om_i)})+C\|\na\pi_h^c v\|_{L^2(\Omega)}\\
&\le&{\rev{\sum^2_{i=1}(\|v_i-\pi_h v_i\|_{\cM_i}+\|w_i\|_{L^2(\Om_i)})+C(\|\na_h(\pi_h^cv-v)\|_{\cM}+\|\na_h v\|_{\cM}}})\\
&\le&C{\rev{\sum^2_{i=1}(\|v_i-\pi_h v_i\|_{H^1(\cM_i)}+\|w_i\|_{H^1(\Om_i)})+C\|\na_h v\|_\cM.}}
\een
Here for $i=1,2$, $\|w\|_{H^1(\cM_i)}^2=\|w\|_{\cM_i}^2+\|\na_h w\|_{\cM_i}^2\ \ \forall w\in H^1(\cM_i)$. This completes the proof by using \eqref{a4} and \eqref{a7}.
\end{proof}

Now we consider the condition number of the stiffness matrix.
We start by introducing the basis functions we use in each element. If $K\in\cM\backslash\cM^\Ga$ is not
an interface element, we will use a set of basis functions which are Lagrangian interpolation functions corresponding to Gauss-Lobatto points. We first recall some facts about spectral method and refer to Bernardi and Maday \cite{Bernardi} for the details.

Let $I=(-1,1)$ and $\{L_i\}^p_{i=0}$ the set of Legendre polynomials of $Q_p(I)$ which is the set of polynomials of degree $p$ in $I$. Let $\{l_i\}^p_{i=0}$ be the set of Lagrangian interpolation functions in $Q_p(I)$ corresponding to the Gauss-Lobatto points $\{\xi_i\}^p_{i=0}$ which are the zeros of $(1-\xi^2)L_p'(\xi)$ in $I$.

Now let $\hat K=I\times I$ and $\{(\xi_i,\xi_j):0\le i,j\le p\}$ be the Gauss-Lobatto grid of $\hat K$. Any function $\hat v\in Q_p(\hat K)$ can be written as $\hat v=\sum^p_{i,j=0}\hat v_{ij}l_i(\hat x_1)l_j(\hat x_2)$. The following important result is proved in Melenk \cite[Proposition 2.8, Theorem 4.1]{Melenk}.

\begin{lem}\label{lem:4.1}
There exists a constant $C$ independent of $p$ such that for any function $\hat v=\sum^p_{i,j=0}\hat v_{ij}l_i(\hat x_1)l_j(\hat x_2)$, there holds
\begin{align*}
& C^{-1}p^{-2}\sum^p_{i,j=0}\hat v^2_{ij}\le\|\hat v\|_{H^1(\hat K)}^2\le Cp\sum^p_{i,j=0}\hat v_{ij}^2,
\end{align*}
and
\begin{align*}
& \|\hat v\|^2_{L^2(\pa \hat K)}\le Cp^{-1}\left(\sum_{i=0,p}\sum^p_{j=0}\hat v_{ij}^2+\sum_{j=0,p}\sum_{i=0}^p\hat v_{ij}^2\right).
\end{align*}
\end{lem}

For any $K\in\cM$, let $F_K:\hat K\to K$ be the one-to-one and surjective affine mapping. Denote $\phi_K^{ij}=\hat\phi_{ij}\circ F_K^{-1}$, where $\hat\phi_{ij}=l_i(\hat x_1)l_j(\hat x_2)$, $0\le i,j\le p$. For any $v\in Q_p(K), v=\sum^p_{i,j=0} v_K^{ij}\phi_K^{ij}$, we have by Lemma \ref{lem:4.1} and the standard scaling argument that
\begin{align}
&C^{-1}p^{-2}\|\bV_K\|_{\ell_2}^2\le\|\na v\|_{L^2(K)}^2+h_K^{-2}\|v\|_{L^2(K)}^2\le Cp\|\bV_K\|_{\ell_2}^2,\label{b7}\\
&\|v\|_{L^2(\pa K)}^2\le Cp^{-1}h_K\|\bV_K\|_{\ell_2}^2,\label{b8}
\end{align}
where $\bV_K=(\bv_0^T,\cdots,\bv^T_p)^T$, $\bv_i=(v_{i_0},\cdots,v_{i_p})^T$ is the coefficient vector corresponding to $v\in Q_p(K)$.

For the interface element $K\in\cM^\Ga$, we have $K_1^{h-\delta_K}\subset K_1$ and $K_2^{h-\delta_K}\subset K_2$. We also have $\hat K_1^{h-\delta_K}\subset \hat K_1$ and $\hat K_2^{h-\delta_K}\subset \hat K_2$,  where $\hat K_i^{h-\delta_K}=F_K^{-1}(K_i^{h-\delta_K})$, $\hat K_i=F_K^{-1}(K_i)$, $i=1,2$. Let $\{\hat\psi^j_{\hat K^h_i}\}^{(p+1)^2}_{j=1}$ the $L^2$-orthonormal basis of $Q_p(\hat K^{h-\delta_K}_i)$, that is, $(\hat\psi^j_{\hat K_i^h},\hat\psi^k_{\hat K_i^h})_{\hat K_i^{h-\delta_K}}=\delta_{kj}$. Denote by $\psi^j_{K_i^h}=p^{-3/2}(\hat\psi^j_{\hat K^h_i}\circ F_K^{-1})$. Then $\{\psi^j_{K_i^{h}}\}^{(p+1)^2}_{j=1}$ is an $L^2$-orthogonal basis of $Q_p(K_i^{h-\delta_K})$, that is,
\beq\label{b4}
(\psi^j_{K_i^h},\psi^k_{K_i^h})_{K_i^{h-\delta_K}}=p^{-3}\frac{|K|}{|\hat K|}\,\delta_{jk}.
\eeq
The scaling constant $p^{-3}$ in \eqref{b4} is important for us to balance the contribution of different basis functions used in interface and non-interface elements in the estimation of the condition number of the stiffness matrix. Now for any $v\in\X_p(\cM)$, $K\in\cM^\Ga$,
\beq\label{yy1}
v|_K=\sum^{(p+1)^2}_{j=1}(v_{K_1}^j\psi^j_{K_1^h}\chi_{K_1}+v^j_{K_2}\psi^j_{K_2^h}\chi_{K_2}):=v_1\chi_{K_1}+v_2\chi_{K_2}.
\eeq
Let $\bV_K=(v^1_{K_1},\cdots,v^{(p+1)^2}_{K_1},v^1_{K_2},\cdots,v_{K_2}^{(p+1)^2})^T$ the coefficient vector corresponding to $v$, then by \eqref{b4} we have
\beq\label{b5}
\|v_1\|_{L^2(K_1^{h-\delta_K})}^2+\|v_2\|_{L^2(K_2^{h-\delta_K})}^2=p^{-3}\frac{|K|}{|\hat K|}\,\|\bV_K\|_{\ell_2}^2.
\eeq
By Lemma \ref{lem:new} we obtain
\beq\label{b6}
Cp^{-3}h_K^2\|\bV_K\|_{\ell_2}^2\le \|v\|_{L^2(K)}^2\le C\Theta_{K}p^{-3}h_K^2\|\bV_K\|_{\ell_2}^2\ \ \ \forall K\in\cM^\Ga.
\eeq

Now, by the construction, any function $v\in\X_p(\cM)$ can be written as
\beq\label{c1}
v=\sum_{K\in\cM\backslash\cM^\Ga}\sum^p_{i,j=0}v_K^{ij}\phi^{ij}_K+\sum_{K\in\cM^\Ga}\sum^{(p+1)^2}_{j=1}(v_{K_1}^j\psi^j_{K_1^h}\chi_{K_1}+v^j_{K_2}\psi^j_{K_2^h}\chi_{K_2}).
\eeq
Let $N=\#\cM$ be the number of elements of the mesh $\cM$, $\{G_1,\cdots,G_{N}\}$ the elements of $\cM$, and $\bV_{G_i}$ the coefficient vector of $v|_{G_i}$, $i=1,\cdots,N$. We denote $\bV=(\bV_{G_1}^T,\cdots,\bV_{G_N}^T)^T$ the vector of coefficients of $v$. The dimension of the vector $\bV$ is $N_p=(p+1)^2N$. We write $\bV=\Phi(v)$, where $\Phi:\X_p(\cM)\to\R^{N_p}$ is the mapping between functions in $\X_p(\cM)$ and their coefficient vectors.

Let $\bV=\Phi(v),\bW=\Phi(w)\in\R^{N_p}$ for $v,w\in\X_p(\cM)$. Then the stiffness matrix $\mathbb{A}=(a_{ij})^{N_p}_{i,j=1}$ is defined by
\ben
 (\mathbb{A}\bV,\bW)_{\ell_2}=a_h(v,w).
\een

Recall that $\Theta=\max_{K\in\cM}\Theta_K$. The following theorem is the main result of this section.

\begin{thm}\label{thm:4.1}
Denote $N^\Ga=\#\cM^\Ga$ the number of elements of $\cM^\Ga$ and $M=\min(N-N^\Ga,N^\Ga)$. Then the following bound of the condition number of the stiffness matrix holds
\ben
\kappa(\mathbb{A})\le C\Theta^2(1+|\ln(h_{\min}^2M)|)\left(p^3(N-N^\Ga)+p^4N^\Ga\right),
\een
where $h_{\min}=\min_{K\in\cM}h_K$ and the constant $C>0$ is independent of the mesh sizes, $p$, and the interface deviations $\eta_K$ for all $K\in\cM^\Ga$.
\end{thm}

We note that $N-N^\Ga$ is the number of non-interface elements. For elliptic equations, it is well-known that the condition number of the stiffness matrix of standard finite element methods grows linearly in terms of the number of elements (see, e.g., Bank and Scott \cite{Bank}). The condition number of the stiffness matrix of the $hp$ finite element method using Gauss-Lobatto shape functions is studied in \cite{Melenk}, which in particular generalizes earlier results that the condition number grows as $O(p^3)$ of the spectral method. Thus the estimate in Theorem \ref{thm:4.1} is optimal in terms of the number of elements and $p$. Our numerical results in Example 1 of section 5 show that the bound is also sharp in terms of the growth factor $\Theta^2$.

\begin{proof} For any $v=v_1\chi_{\Om_1}+v_2\chi_{\Om_2}\in\X_p(\cM)$, denote $w=(\pi_h v_1)\chi_{\Om_1}+(\pi_h v_2)\chi_{\Om_2}$ and $\bW=\Phi(w)$ the coefficient vector corresponding to $w$. By \eqref{b7}, \eqref{b6} and Lemma \ref{lem:2.1} we know that
\ben
\|\bV-\bW\|_{\ell_2}^2&\le&Cp^2\sum_{K\in\cM\backslash\cM^\Ga}(\|\na(v-w)\|_{L^2(K)}^2+{\rev{h_K^{-2}}}\|v-w\|_{L^2(K)}^2)\\
& &+\,C\sum_{K\in\cM^\Ga}p^3h_K^{-2}\|v-w\|_{L^2(K)}^2\\
&\le&C p^2\|ph^{-1/2}\lj v\rj\|_{\cE^{\rm side}_1\cup\cE^{\rm side}_2}^2.
\een
Thus by the triangle inequality
\begin{align}
\|\bV\|_{\ell_2}^2
\le C p^2\|ph^{-1/2}\lj v \rj \|_{\cE^{\rm side}_1\cup\cE^{\rm side}_2}^2+2\|\bW\|_{\ell_2}^2.\label{d1}
\end{align}
Again by \eqref{b7}, \eqref{b6} we have
\ben
{\rev{\|\bW\|_{\ell_2}^2\le Cp^2\sum_{K\in\cM\backslash\cM^\Ga} \big(\|\na w\|_{L^2(K)}^2+h_K^{-2}\|w\|_{L^2(K)}^2\big)+C\sum_{K\in\cM^\Ga}p^3h_K^{-2}\|w\|_{L^2(K)}^2.}}
\een
Now we use an argument in \cite{Bank}. By H\"older inequality, for any $r\ge 2$,
\ben
\sum_{K\in\cM\backslash\cM^\Ga}h_K^{-2}\|w\|_{L^2(K)}^2
&\le&C\sum_{K\in\cM\backslash\cM^\Ga}h_K^{-4/r}\|w\|_{L^r(K)}^2\\
&\le&{\rev{C\left(\sum_{K\in\cM\backslash\cM^\Ga} h_K^{-4/(r-2)}\right)^{\frac{r-2}r}\|w\|_{L^r(\Om_1\cup\Om_2)}^2}}\\
&\le&{\rev{Ch_{\min}^{-4/r}(N-N^\Ga)^{\frac{r-2}r}\|w\|_{L^r(\Om_1\cup\Om_2)}^2}}.
\een
Similarly,
\ben
{\rev{\sum_{K\in\cM^\Ga}p^3h_K^{-2}\|w\|_{L^2(K)}^2\le Cp^3h_{\rm min}^{-4/r}(N^\Ga)^{\frac{r-2}r}\|w\|_{L^r(\Om_1\cup\Om_2)}^2}}.
\een
Therefore,
\ben
\|\bW\|_{\ell_2}^2&\le&{\rev{Cp^2\|\na w\|_{\cM\backslash\cM^\Ga}^2+C(p^2(N-N^\Ga)+ p^3N^\Ga)h_{\rm min}^{-4/r}M^{-2/r}\|w\|_{L^r(\Om_1\cup\Om_2)}^2}}\\
&\le&{\rev{C(p^2(N-N^\Ga)+ p^3N^\Ga)(h^2_{\rm min}M)^{-2/r}r\|w\|_{H^1(\Om_1\cup\Om_2)}^2}},
\een
where we have used the embedding inequality, $\|w\|_{L^r(D)}\le Cr^{1/2}\|w\|_{H^1(D)}$ for any $w\in H^1(D)$, $r\ge 1$, on any Lipschitz domain $D$. Notice that for any $\zeta>0$, $\zeta^{-2/r}=e^{-2\ln\zeta/r}=e^{-2}$ if $r=\ln\zeta$, by taking $r=\max(2,|\ln(h_{\rm min}^2M)|)$ we obtain
\beq\label{d2}
\|\bW\|_{\ell_2}^2\le {\rev{C(p^2(N-N^\Ga)+ p^3N^\Ga)(1+|\ln(h_{\rm min}^2M)|)\|w\|_{H^1(\Om_1\cup\Om_2)}^2}}.
\eeq
By Lemma \ref{lem:2.1} and the discrete Poincar\'e inequality in Lemma \ref{lem:2.3}
\ben
\|w\|_{H^1(\Om_1\cup\Om_2)}^2&\le& {\rev{2\|w-v\|_{H^1(\Om_1\cup\Om_2)}^2+2(\|\nabla_h v\|_{L^2(\Om)}^2+\|v\|_{L^2(\Omega)}^2)}}\\
&\le&{\rev{C(\|ph^{-1/2}\lj v\rj \|_{\cE_1^{\rm side}\cup\cE_2^{\rm side}}^2+\|\nabla_h v\|_{L^2(\Om)}^2+\|v\|_{L^2(\Omega)}^2)}}\\
&\le& {\rev{Ca_h(v,v).}}
\een
This yields by \eqref{d1}-\eqref{d2} that
\beq\label{d3}
\|\bV\|_{\ell_2}^2\le C (p^2(N-N^\Ga)+p^3N^\Ga)(1+|\ln(h_{\rm min}^2M)|)a_h(v,v).
\eeq
On the other hand, since $a_h(v,v)\le C\|v\|_{\rm DG}^2$, we have
\be
a_h(v,v)&\le&C\sum_{K\in\cM\backslash\cM^\Ga}\left(\|\na v\|_{L^2(K)}^2+\Theta_K\|ph^{-1/2}v\|_{L^2(\pa K)}^2\right)\nn\\
& &+\sum_{K\in\cM^\Ga}\sum^2_{i=1}\left(\|\na v_i\|_{L^2(K_i)}^2+\Theta_K\|ph^{-1/2}v_i\|_{L^2(\pa K_i)}^2+\|p^{-1}h^{1/2}\na_T v_i\|_{L^2(\Ga_K)}^2\right)\nn\\
&:=&{\rm I}+{\rm II}.\label{d4}
\ee
By \eqref{b7}-\eqref{b8}
\beq\label{d5}
{\rm I}\le C\Theta\sum_{K\in\cM\backslash\cM^\Ga}\left(\|\na v\|_{L^2(K)}^2+\|ph^{-1/2}v\|_{L^2(\pa K)}^2\right)\le C\Theta p\sum_{K\in\cM\backslash\cM^\Ga}\|\bV_K\|_{\ell_2}^2.
\eeq
For $K\in\cM^\Ga$, by Lemma \ref{lem:2.1} and \eqref{b6}, for $i=1,2$,
\ben
\|\na v_i\|_{L^2(K_i)}^2\le C\Theta_Kp^4h_K^{-2}\|v_i\|_{L^2(K_i)}^2\le C\Theta_K^2p\|\bV_K\|_{\ell_2}^2.
\een
By \eqref{a10}, Lemma \ref{lem:2.1}, the $hp$ trace inequality, and inverse estimate
\ben
\|v_i\|_{L^2(\pa K_i)}^2&\le&C\|v_i\|_{L^2(K_i)}\|\na v_i\|_{L^2(K_i)}+C\|v_i\|_{L^2(\pa K_i^h)}^2\\
&\le&C\Theta_{K}\|v_i\|_{L^2(K_i^{h-\delta_K})}\|\na v_i\|_{L^2(K_i^{h-\delta_K})}+Cp^2h_K^{-1}\|v_i\|_{L^2(K_i^h)}^2\\
&\le&C\Theta_{K} p^2h_K^{-1}\|v_i\|_{L^2(K_i^{h-\delta_K})}^2,
\een
where we used the fact $\|v_i\|_{L^2(K_i^h)}^2\le C\Theta_{K} \|v_i\|_{L^2(K_i^{h-\delta_K})}^2$, which follows directly from Lemma \ref{lem:2.0}, in the last inequality. Thus by \eqref{b5}
\ben
\Theta_K\|ph^{-1/2}v_i\|_{L^2(\pa K_i)}^2\le C\Theta_K^2p^4h_K^{-2}\|v_i\|_{L^2(K_i^{h-\delta_K})}^2\le C\Theta_K^2 p\|\bV_K\|_{\ell_2}^2.
\een
Similarly, one can prove $\|p^{-1}h^{1/2}\na_T v_i\|_{L^2(\Ga_K)}^2\le C\Theta_K^2 p\|\bV_K\|_{\ell_2}^2$. Therefore, we have
\beq\label{d6}
{\rm II}\le C\Theta^2 p\sum_{K\in\cM^\Ga}\|\bV_K\|_{\ell_2}^2.
\eeq
Combining \eqref{d4}-\eqref{d6} we obtain
\ben
a_h(v,v)\le C\Theta^2 p\|\bV\|_{\ell_2}^2.
\een
This completes the proof by using \eqref{d3}.
\end{proof}

To conclude this section, we remark that since $\Theta_K=\mathsf{T}(\frac{1+3\eta_K}{1-\eta_K})^{4p+3}$, Theorem \ref{thm:4.1} indicates that to control the condition number of the stiffness matrix, one should choose $\eta_K\ll 1$, that is, one should have the interface being well resolved by the mesh.

\section{Numerical examples}\label{sec_numeric}
In this section we provide some numerical examples to verify our theoretical results. In order to construct the orthogonal polynomials on the polygons $\hat{K}_i^{h-\delta_K}$ for the interface elements $K$, we adopt the Gram-Schmidt process starting from the basis functions of $Q_p(\hat{K})$ which are the Lagrange interpolation polynomials through the Gauss-Lobatto integration points on $\hat{K}$. The details can be found in Sommariva and Vianello \cite{Sommariva2017MCS}. The algorithms are implemented in MATLAB on a workstation with Intel(R) Core(TM) i9-10885H CPU 2.40GHz and 64GB memory.

\begin{exmp}\label{example_cond}
In this example we show that the growth factor $\Theta^2$ in the bound of the condition number of the stiffness matrix in Theorem \ref{thm:4.1} is sharp. For this purpose, we consider the case of one interface element.
Let $K=(-2,2)^2$ and the interface $\Gamma=\{(x(t),y(t))\in \mathbb{R}^2:t\in(-\frac{\sqrt{2}}{2},\frac{\sqrt{2}}{2})\}$, where $x(t)$ and $y(t)$ are defined as follows:
\begin{align*}
&x(t)=\sqrt{2}\cos(\alpha+\frac{\pi}{4})t+\sqrt{2}\sin(\alpha+\frac{\pi}{4})(100t^3-\beta t)-1,\\
&y(t)=-\sqrt{2}\sin(\alpha+\frac{\pi}{4})t+\sqrt{2}\cos(\alpha+\frac{\pi}{4})(100t^3-\beta t)-1,
\end{align*}
where $\cos(\alpha)=\frac{1}{\sqrt{\mu^2+1}}$, $\sin(\alpha)=\frac{\mu}{\mu^2+1}$, $\beta=\frac{100}{\sqrt{\mu^2+1}}-\mu$ and $\mu=3.8$.
\end{exmp}

The domain and the interface are shown in Fig.\ref{fig_one_element} (left) in which $K_1^{h-\delta_K}=\Delta AE'F'$ and $K_2^{h-\delta_K}$ is the polygon with vertices $F'',E'',B,C,D$. The interface deviation is $\eta_K=\frac{200}{3\sqrt{3}(\mu^2+1)^2}\approx 0.16$. We first consider the condition number of the mass matrix to verify our analysis in Lemma \ref{lem:new}. For $v\in\X_p(K)$, in the notation of \eqref{yy1}, the mass matrix $\mathbb{M}\in\R^{(p+1)^2\times(p+1)^2}$ is defined as $(\mathbb{M}\bV_K,\bW_K)_{\ell_2}=(v,w)_K\ \ \forall v,w\in\X_p(K)$. Then \eqref{b6} implies that the condition number $\kappa(\mathbb{M})\le C\Theta$ for some constant $C$ independent of $p$ and $\eta_K$. We plot $\kappa(\mathbb{M})$ vs. $\Theta$ via different degrees of polynomials with loglog scaling in Fig.\ref{fig_one_element} (right). It is clear that the condition number of $\mathbb{M}$ grows as $\Theta$ which agrees with our theoretical bound.

We plot the curve $\kappa(\mathbb{A})$ vs. $\Theta^2 p^4$ via different degrees of polynomials with loglog scaling in
Fig.\ref{fig_one_element_small_eta} (left).  We observe that the condition number grows as $\Theta^2p^4$ which confirms our analysis in Theorem \ref{thm:4.1}. We also observe that the $\kappa(\mathbb{A})$ increases very fast with the increase of polynomial degree. One can reduce the interface deviation to reduce the $\kappa(\mathbb{A})$. We change $\mu$ to reduce $\eta_K$ such that $\eta_K\le \frac{0.1}{p(p+1)}$ and plot the curve $p^4$ vs. $\kappa(\mathbb{A})$ in Fig.\ref{fig_one_element_small_eta} (right). We can find the $\kappa(\mathbb{A})$ is significantly reduced and the $\kappa(\mathbb{A})$ has $p^4$ increasing rates.

\begin{figure}[!ht]
\centering
\includegraphics[width=0.35\textwidth]{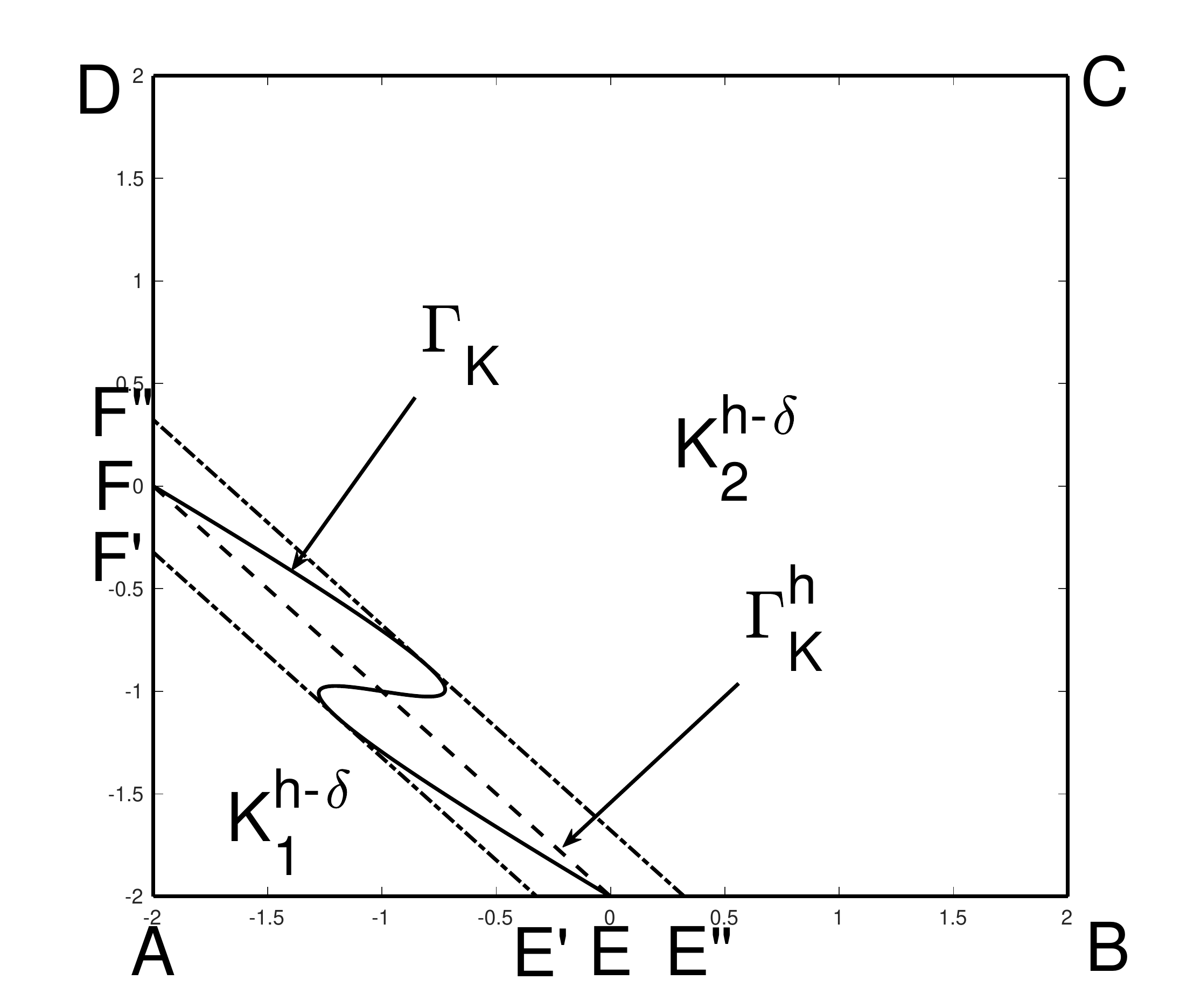}
\includegraphics[width=0.4\textwidth]{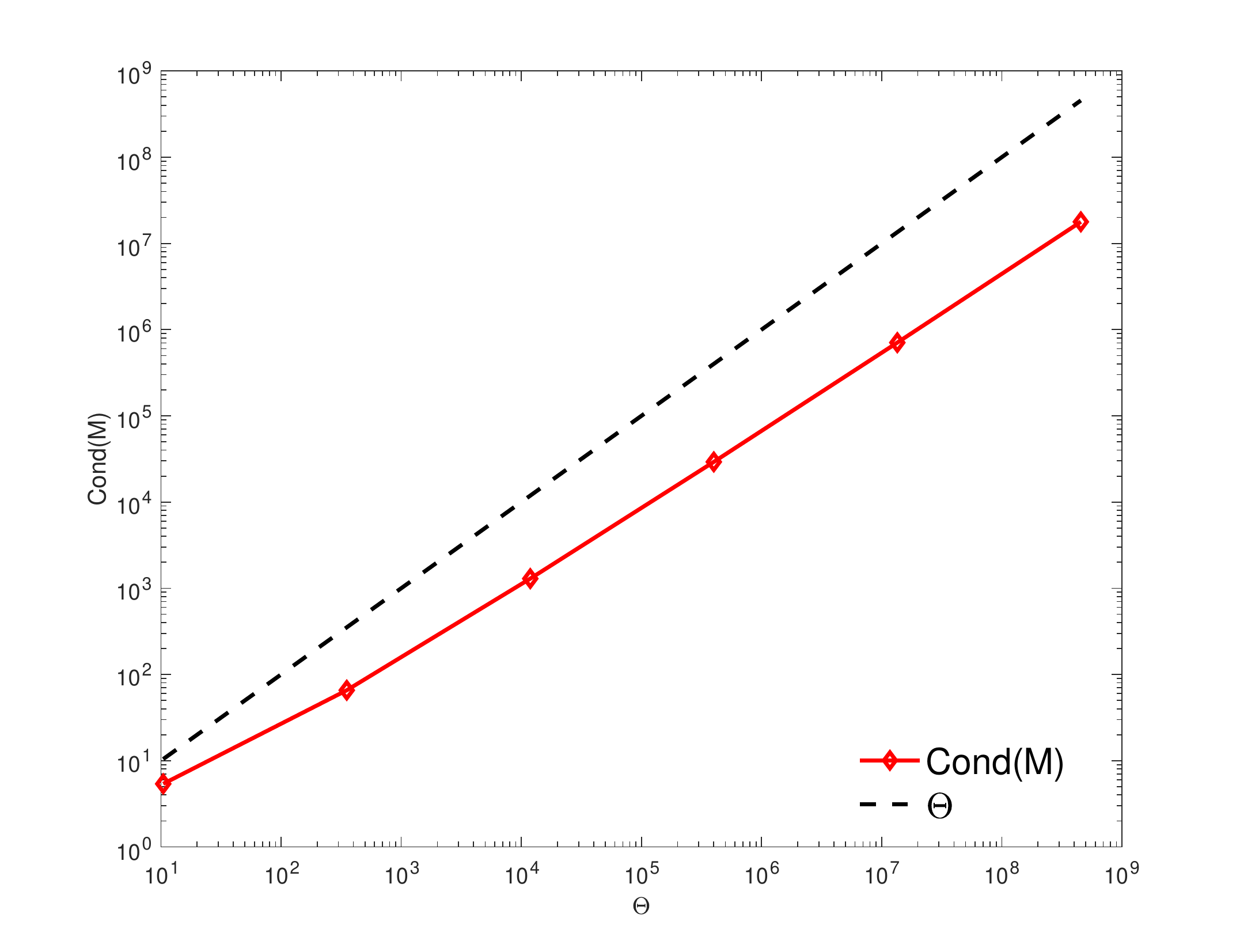}
\caption{Example \ref{example_cond}: The geometry setting of Example \ref{example_cond} (left) and the growth rate of the condition number of the mass matrix (right).} \label{fig_one_element}
\end{figure}

\begin{figure}[!ht]
\centering
\includegraphics[width=0.4\textwidth]{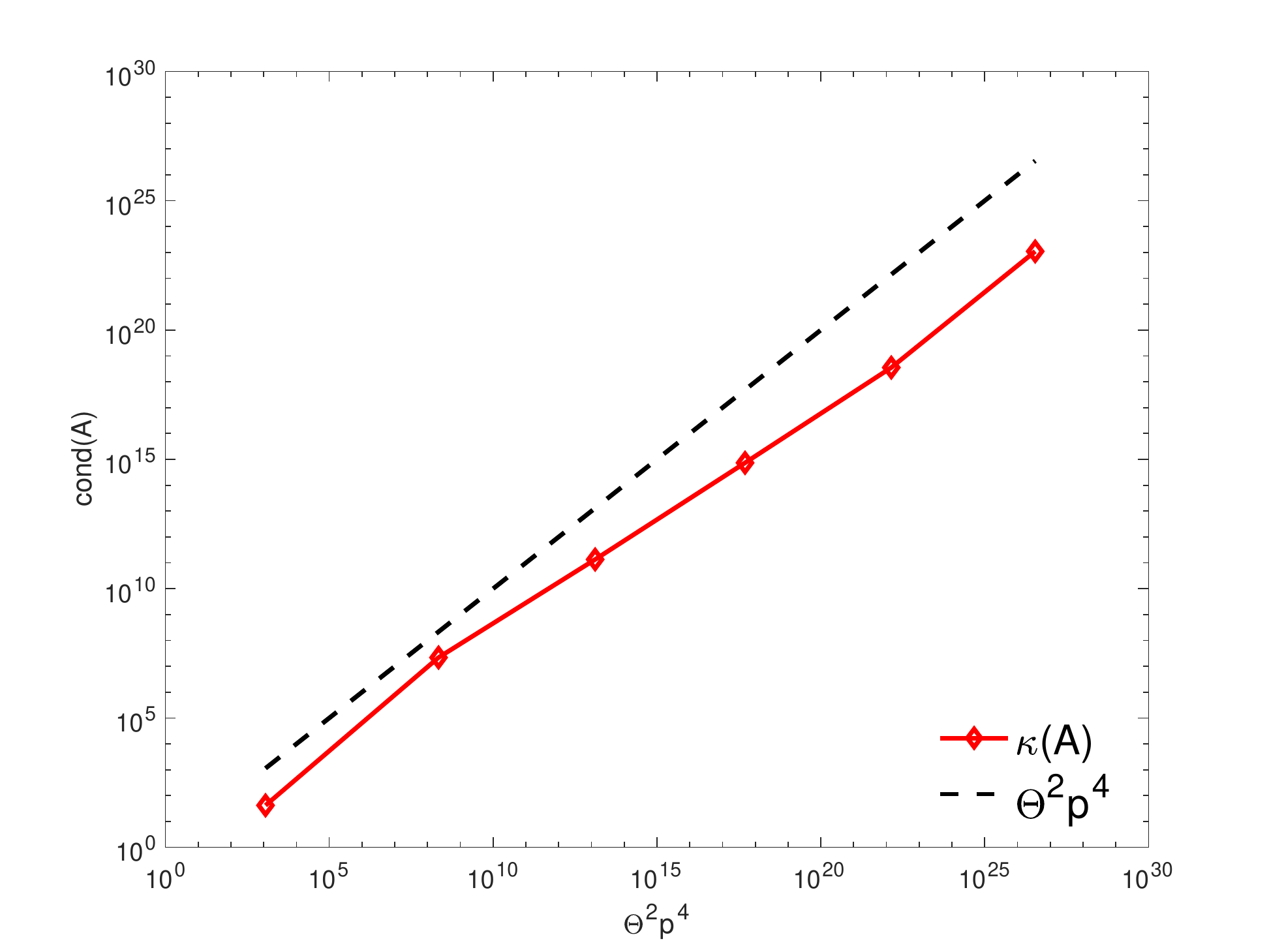}
\includegraphics[width=0.4\textwidth]{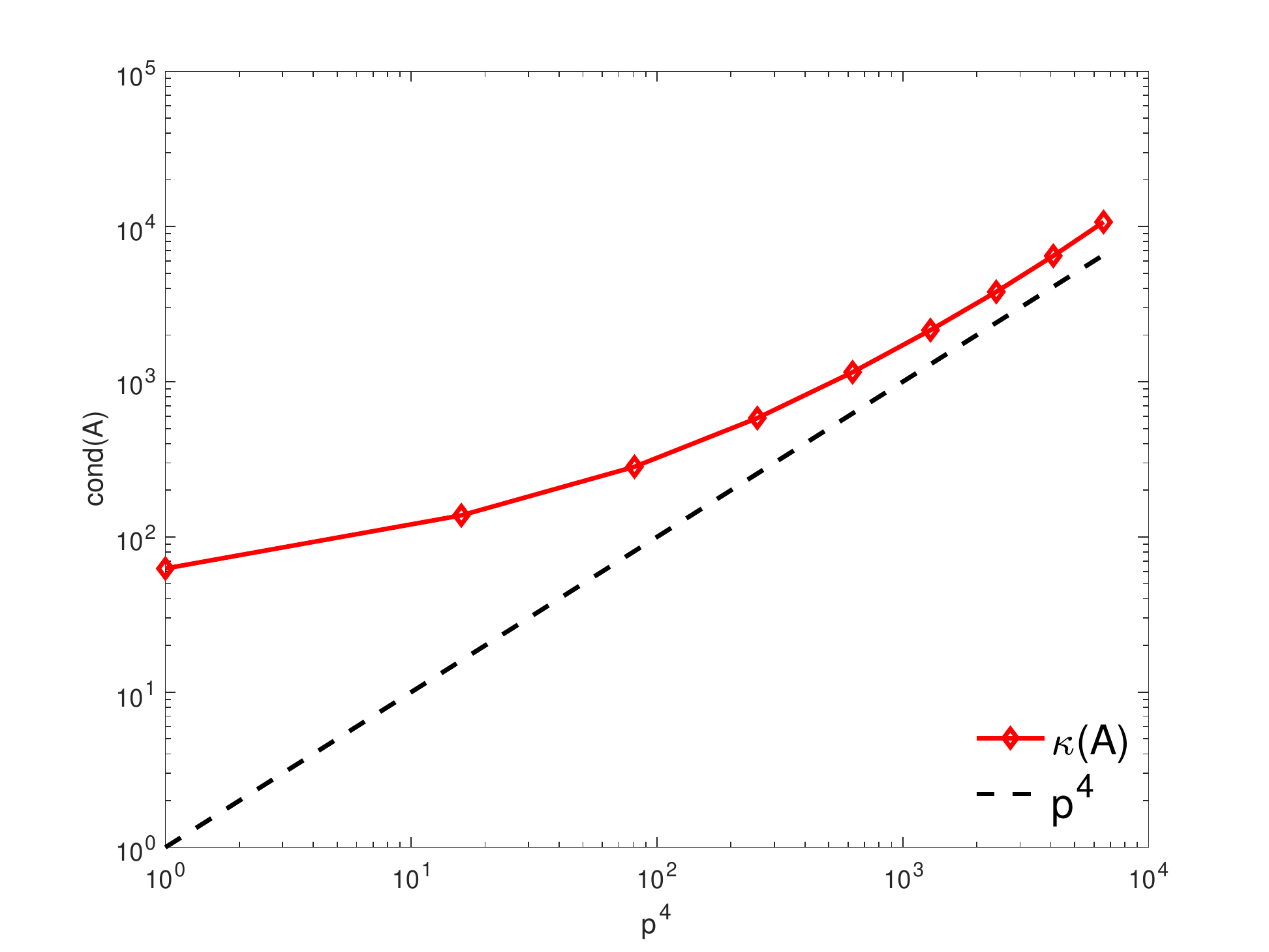}
\caption{Example \ref{example_cond}: The growth rate of the condition number of $\mathbb{A}$ with $\eta_K=0.16$ (left) and the condition number of $\mathbb{A}$ with $\eta_K\le \frac{0.1}{p(p+1)}$ (right).} \label{fig_one_element_small_eta}
\end{figure}

This example shows clearly the importance of reducing the interface deviation to control the condition number of the stiffness matrix.
In the following we always require
\beq\label{yy}
\max_{K\in\cM}\eta_K\le\frac{0.1}{p(p+1)},
\eeq
which is stronger than that in Assumption (H2). The finite element meshes in our following numerical examples are constructed as follows.

\noindent\rule{\textwidth}{0.35mm}
\noindent{\bf{Algorithm 7:}} {The algorithm for generating the induced mesh satisfying Assumption (H3) and \eqref{yy}}

\vspace{-0.2cm}
\noindent\rule{\textwidth}{0.35mm}

{\bf{Input:}} A uniform initial Cartesian mesh $\cT_0$ of mesh size $h$

{\bf{Output:}} The induced mesh $\cM={\rm Induced}(\mathfrak{C})$

$1^\circ$ Set $\cT=\cT_0$;

$2^\circ$ Refine the elements of $\cT$ near the interface by quad refinements to generate a Cartesian mesh (still denoted by) $\cT$ with possible handing nodes such that all interface elements of $\cT$ form an admissible chain $\mathfrak{C}$;

$3^\circ$ Call the refinement procedure in \cite[\S 6.3]{Bonito} such that $\cT$ satisfies Assumption (H3);

$4^\circ$ Use Algorithm 6 to generate an induced mesh $\cM={\rm Induced}(\mathfrak{C})$;

$5^\circ$ If the interface elements in $\cM$ do not satisfy \eqref{yy},  release all merged elements in $\mathfrak{C}$, go to $2^\circ$.

\vspace{-0.2cm}
\noindent\rule{\textwidth}{0.35mm}

We remark that after step $2^\circ$ in Algorithm 7, the interface elements are of the same size which is smaller than the sizes
of non-interface elements. Thus when implementing the refinement procedure in \cite[\S 6.3]{Bonito} in our situation, only non-interface elements
are refined and consequently, the interface elements still form an admissible chain.

%
%

\begin{exmp}\label{example1}
Let the interface $\Gamma$ be the circle centered at $(0,0)^T$ with radius $r_0=1.1$. We set $\Omega=(-2,2)^2$, $\Omega_1=\{(x,y)\in \mathbb{R}^2:\sqrt{x^2+y^2}<r_0\}$ and $\Omega_2=\Omega\setminus\bar{\Omega}_1$. Set $a_1=10$ and $a_2=1$. The right-hand side $f$ and boundary condition $g$ are computed such that the exact solution is
\begin{align*}
u(x,y)=\left\{\begin{array}{ll}
e^{x^2+y^2-r_0^2}+10r_0^2-1+(x^2+y^2-r_0^2)^2\sin(2\pi x)\sin(2\pi y) & \text{in } \Omega_1,\\
10(x^2+y^2)+(x^2+y^2-r_0^2)^2\sin(2\pi x)\sin(2\pi y) &\text{in }\Omega_2.
\end{array}\right.
\end{align*}
\end{exmp}

\begin{table}[!ht]\centering
\caption{Example \ref{example1}: numerical errors $\|u-U\|_{DG}$ and orders for $p=1,2,3,4,5$.}\label{tab1}
\resizebox{\textwidth}{!}{
\begin{tabular}{|c|cc|cc|cc|cc|cc|}
  \hline
 &\multicolumn{2}{|c|}{$p=1$}&\multicolumn{2}{|c|}{$p=2$}&\multicolumn{2}{|c|}{$p=3$}&\multicolumn{2}{|c|}{$p=4$}&\multicolumn{2}{|c|}{$p=5$}\\\hline
  $h$  &  error & order & error & order & error & order & error & order & error & order\\ \hline
 $1/4$    & 1.13E+00 & --   & 4.00E-01    & --   & 1.20E-01 &  --  &   3.21E-02   & -- &   2.09E-03   & --   \\
 $1/8$   & 6.72E-01 & 0.75 & 1.08E-01    &1.89  	& 2.01E-02 & 2.58 & 1.55E-03   & 4.37 &   1.62E-04   & 3.69 \\
 $1/16$   &  3.57E-01& 0.91 &2.89E-02     & 1.90 	& 2.49E-03 & 3.01 & 1.03E-04   &  3.91 &   5.18E-06   & 4.97\\
 $1/32$   & 1.79E-01 & 0.99     & 7.32E-03    &1.98  	& 3.12E-04 & 3.00 & 6.56E-06   &  3.98&   1.62E-07   & 5.00\\\hline
\end{tabular}
}
\end{table}
In Table \ref{tab1}, we show the errors $\|u-U\|_{DG}$ and the corresponding convergence orders for $p=1,2,3,4,5$.
We clearly observe the optimal $p$-th order convergence and the superior performance of high order methods. Fig. \ref{fig_mesh_one_circle} shows the induced mesh when $h=1/4$ and the corresponding numerical solution.

\begin{figure}[!ht]
\centering
\includegraphics[width=0.3\textwidth]{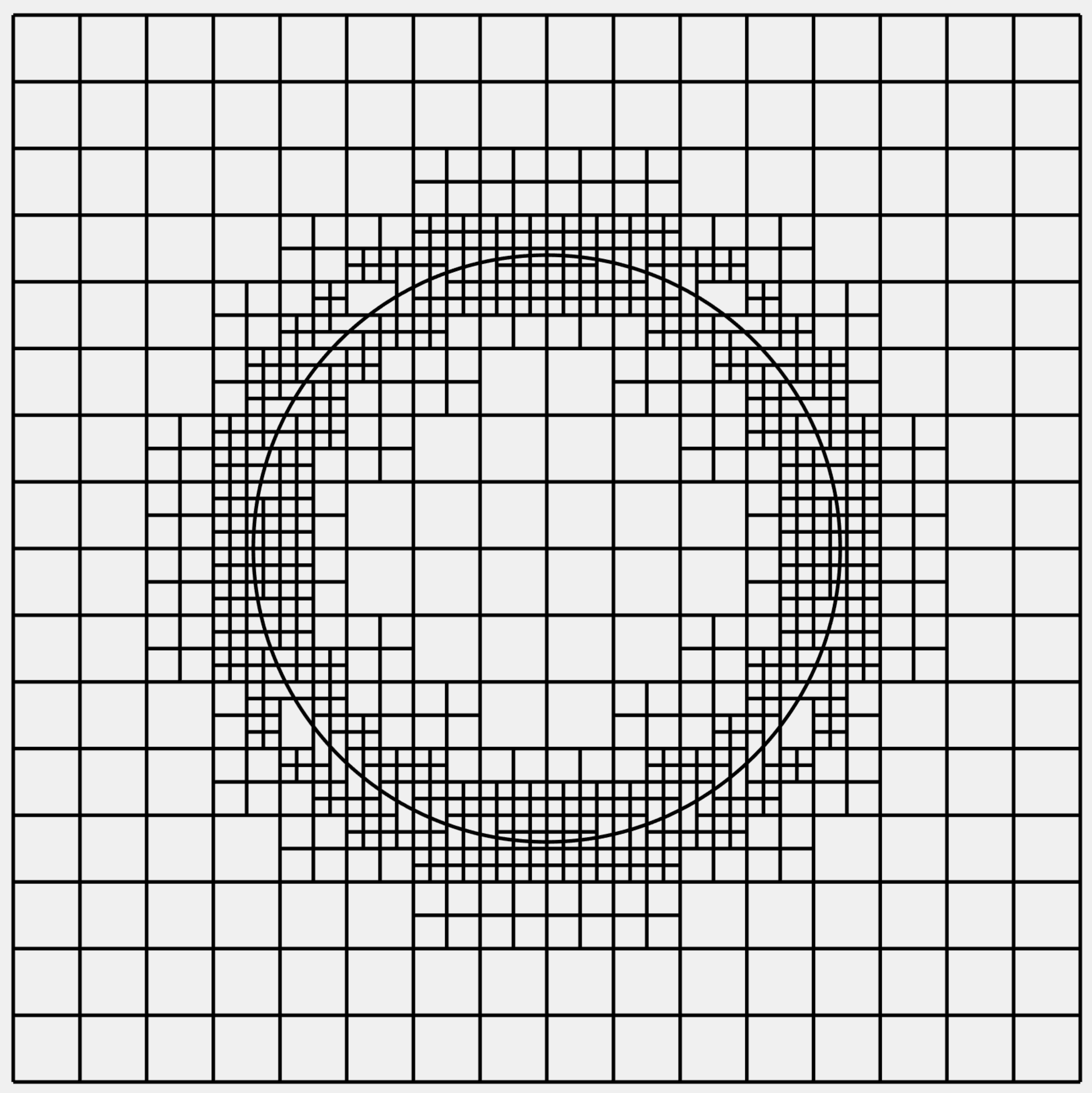}
\includegraphics[width=0.45\textwidth]{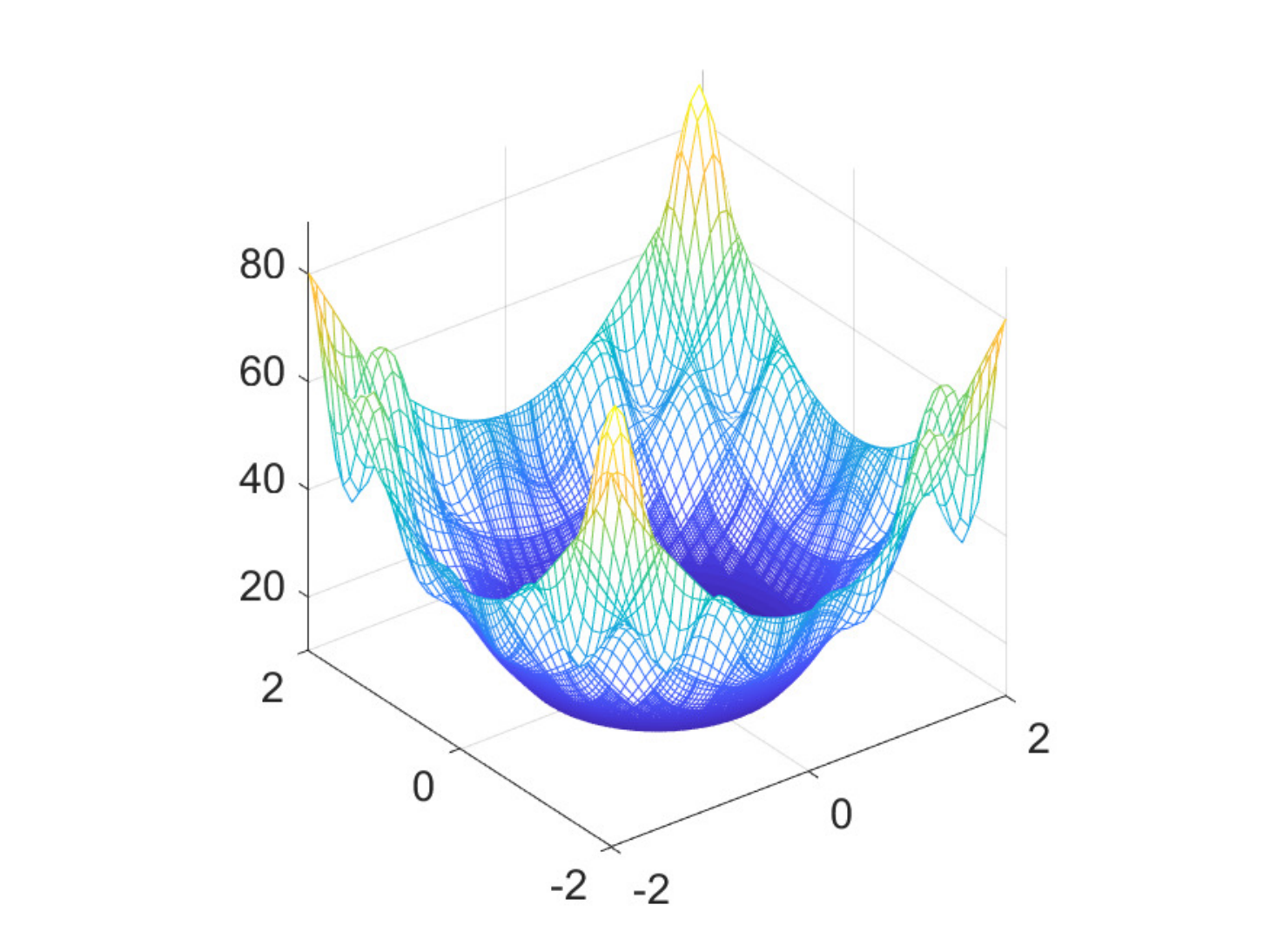}
\caption{Example \ref{example1}: The induced mesh of  $940$ elements when $h=1/4$ (left) and the corresponding numerical solution (right). }\label{fig_mesh_one_circle}
\end{figure}

\begin{exmp}\label{example2}
In this example we consider geometrically more complex interface. Let the interface $\Gamma$ be defined as follows:
\begin{align*}
\Gamma=\{(x,y)\in \mathbb{R}^2:r=\frac{2}{9}(3+4^{\sin(5\theta)})\},
\end{align*}
where $(r,\theta)$ are the polar coordinates. The domain $\Omega$ is divided to $\Omega_1$ and $\Omega_2$ by $\Gamma$, that is,
\begin{align*}
&\Omega_1=\{(x,y)\in (-2,2)^2: r<\frac{2}{9}(3+4^{\sin(5\theta)})\},\\
&\Omega_2=\{(x,y)\in (-2,2)^2: r>\frac{2}{9}(3+4^{\sin(5\theta)})\}.
\end{align*}
We set $a_1=10$, $a_2=1$, the right-hand side $f=1$, and the boundary condition $g=0$.
\end{exmp}

The exact solution of this example is unknown. We use the a posteriori error estimate in \cite{CLX} to measure the accuracy of computation. In Table \ref{tab2}, we observe the optimal $p$-th order convergence. The induced mesh when $h=1/4$ is shown in Fig. \ref{fig_mesh_exmp_five_stars} which has $2654$ elements. The discrete solution is depicted in Fig. \ref{fig_discrete_solution}.

\begin{table}[!ht]\centering
\caption{Example \ref{example2}: A posterior error estimates and the convergence orders for $p=1,2,3,4,5$.}\label{tab2}
\resizebox{\textwidth}{!}{
\begin{tabular}{|c|cc|cc|cc|cc|cc|}
  \hline
 &\multicolumn{2}{|c|}{$p=1$}&\multicolumn{2}{|c|}{$p=2$}&\multicolumn{2}{|c|}{$p=3$}&\multicolumn{2}{|c|}{$p=4$}&\multicolumn{2}{|c|}{$p=5$}\\\hline
  $h$  &  error & order & error & order & error & order & error & order & error & order\\ \hline
 $1/4$    & 1.38E+00 & --   & 1.08E-01   & --   & 3.96E-02 &  --  & 2.85E-03     & -- &1.02E-04     & --   \\
 $1/8$   & 7.31E-01 & 0.92&  2.93E-02  &1.88 	&5.13E-03  &2.95  & 1.83E-04  &3.96 & 3.35E-06    &4.93 \\
 $1/16$   &  3.79E-01& 0.95 & 8.13E-03   &1.85 	&6.46E-04  &2.99  &1.15E-05   &3.99   &1.08E-07     &4.95\\
 $1/32$   & 1.90E-01 & 0.99     &2.09E-03    &1.96 	& 8.12E-05 &2.99 &7.25E-07 &3.99  & 3.41E-09    &4.99 \\
 \hline
\end{tabular}
}
\end{table}

\begin{figure}[!ht]
\centering
\begin{minipage}[c]{0.35\textwidth}
\includegraphics[width=0.9\textwidth, height = 0.9\textwidth]{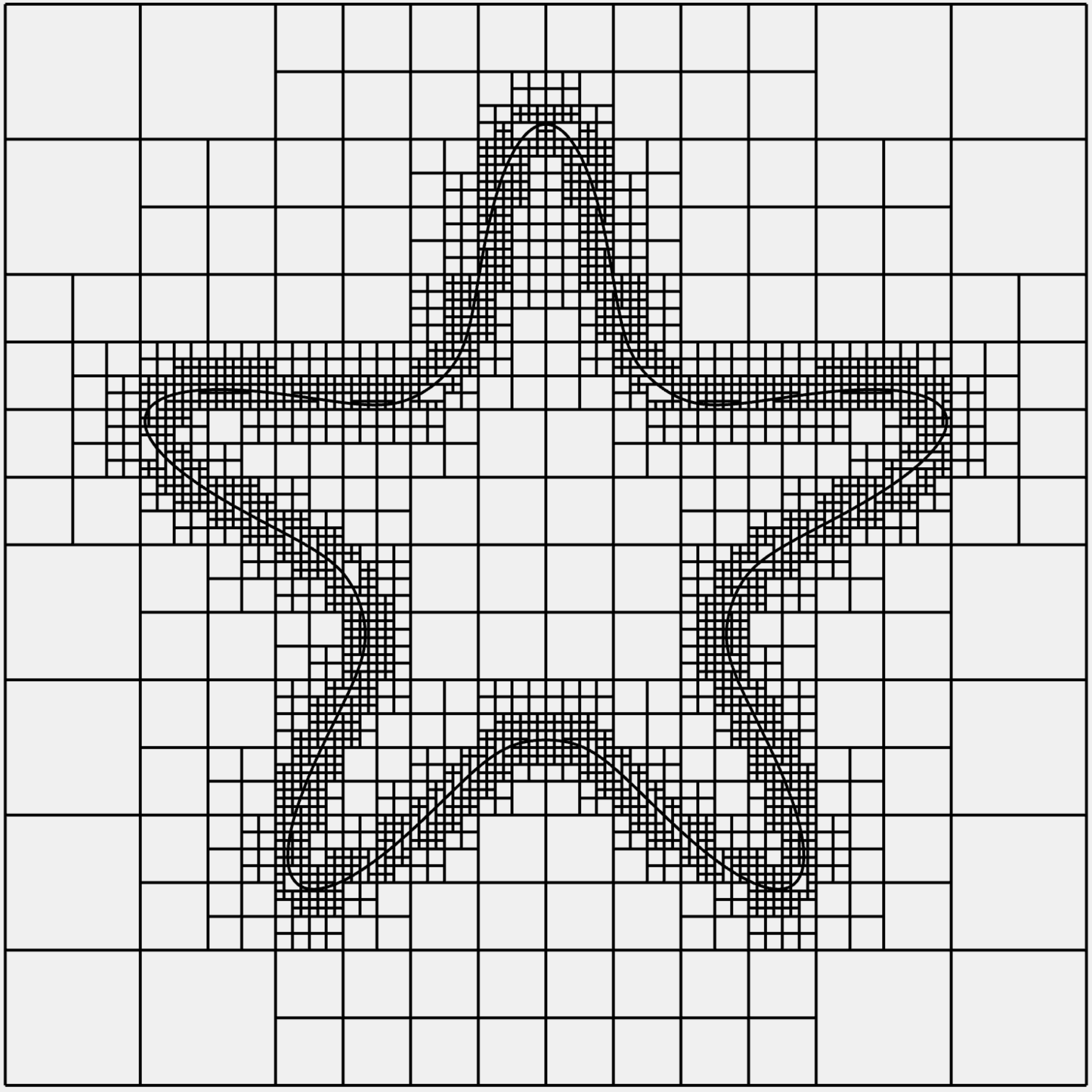}
\end{minipage}
\begin{minipage}[c]{0.35\textwidth}
\includegraphics[width=0.9\textwidth, height = 0.9\textwidth]{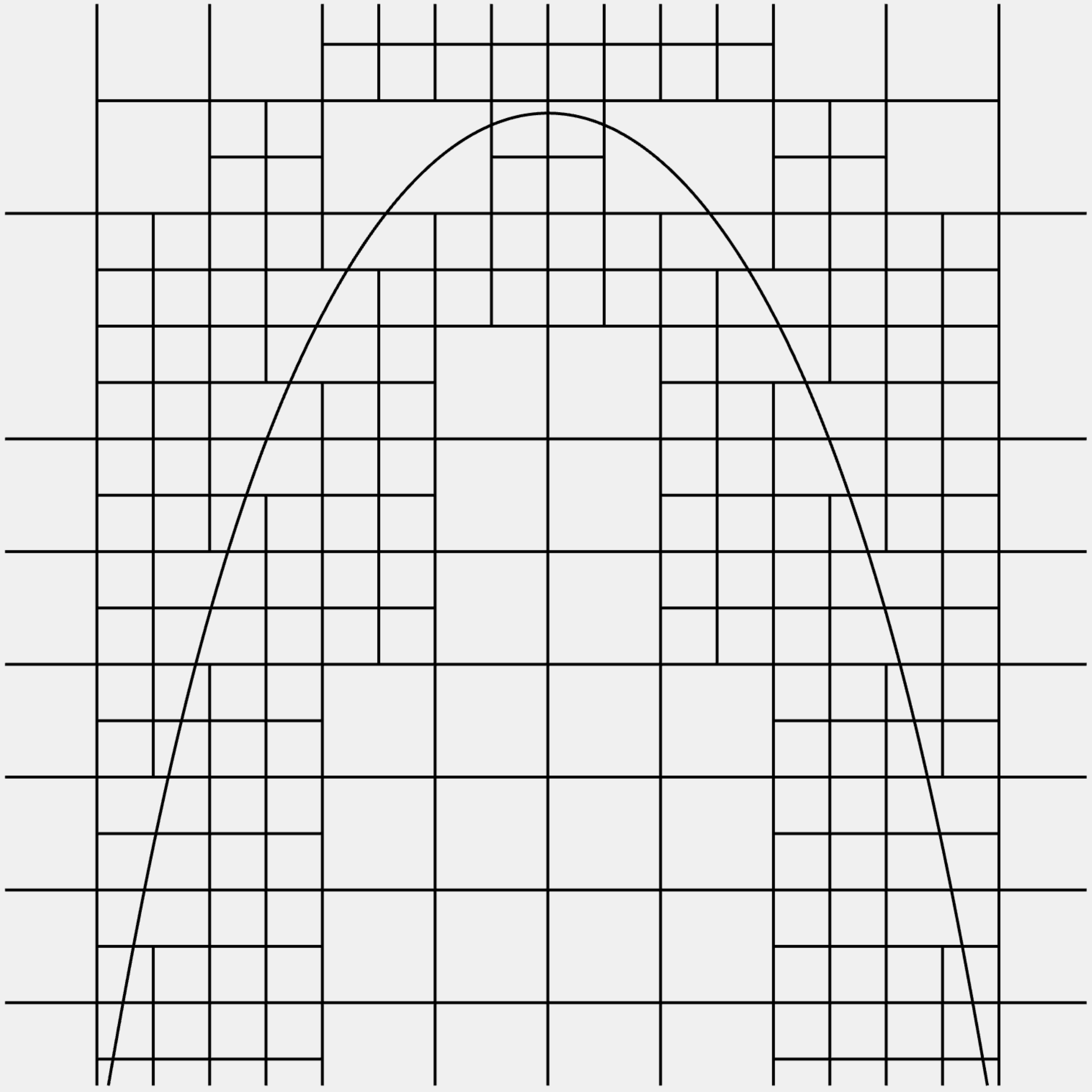}
\end{minipage}
\caption{Example \ref{example2}: The induced mesh of $2654$ elements when $h=1/4$ (left) and the corresponding zoomed local mesh (right). } \label{fig_mesh_exmp_five_stars}
\end{figure}
\begin{figure}[!ht]
\centering
\includegraphics[width=0.5\textwidth]{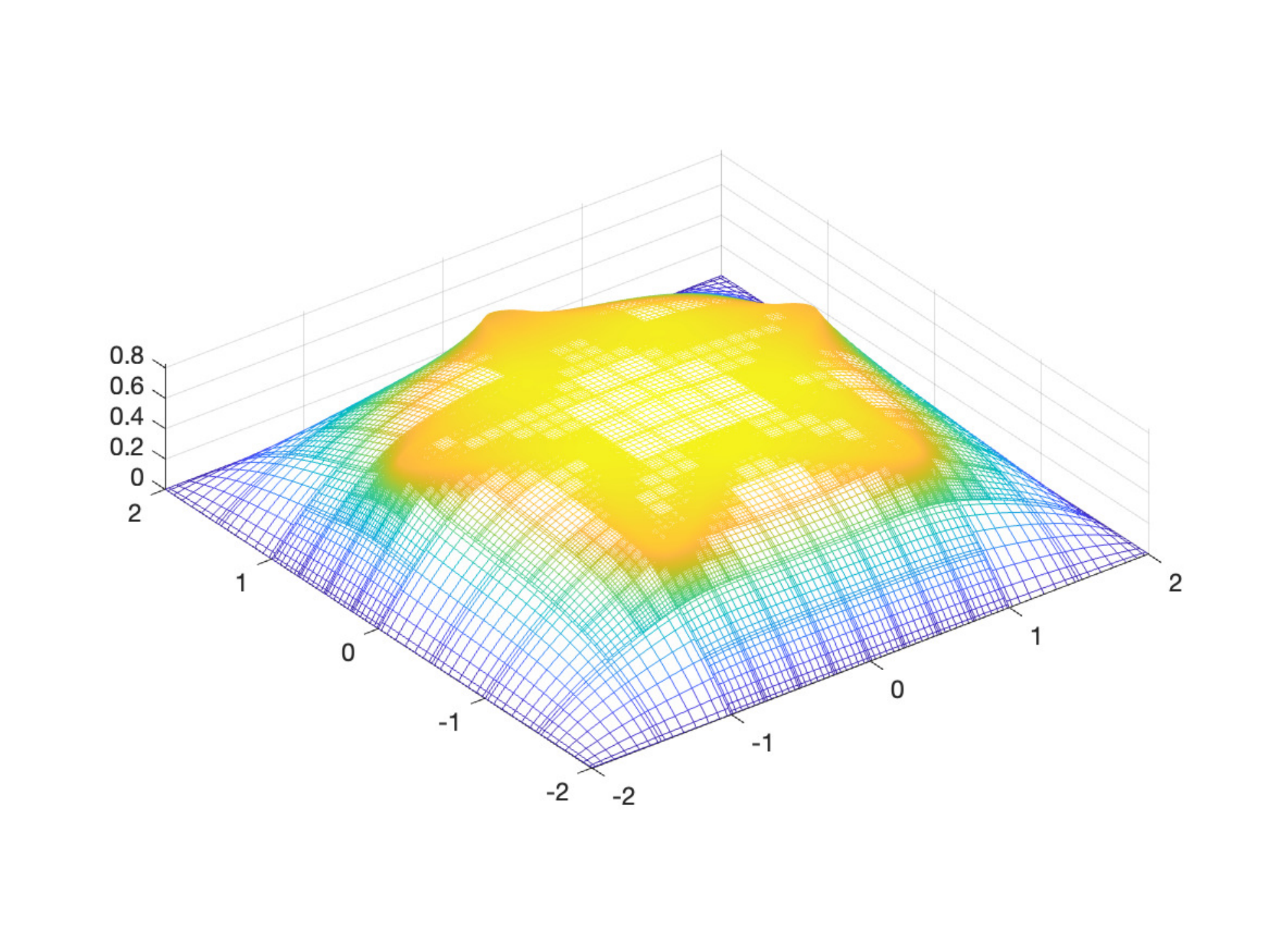}
\caption{Example \ref{example2}: The discrete solution on the mesh of $2654$ elements. } \label{fig_discrete_solution}
\end{figure}

%
%
\bigskip{\noindent \bf{Acknowledgement}}

The authors are grateful to Haijun Wu in Nanjing University for inspiring discussions.

\end{document}